\documentclass{amsart}
\usepackage{amsmath, amsthm, amssymb, latexsym, graphics}

\newcommand{\ds}{\displaystyle}
\newcommand{\fr}[1]{\mathfrak{#1}}
\newcommand{\ca}[1]{\mathcal{#1}}
\newcommand{\bk}{{{\boldsymbol{k}}}}
\newcommand{\ac}{\operatorname{ac}}

\newcommand{\Z}{\mathbb{Z}}
\newcommand{\af}{\mathbf{f}}
\newcommand{\ag}{\mathbf{g}}
\newcommand{\cha}{\operatorname{char}}
\newcommand{\val}{\operatorname{v}}
\newcommand{\res}{\operatorname{r}}

\def \bsigma{{\boldsymbol{\sigma}}}

\def \i{{\boldsymbol{i}}}
\def \j{{\boldsymbol{j}}}

\def \l{{\boldsymbol{l}}}

\def \f{\operatorname{f}}
\def \v{\operatorname{v}}
\def \re{\operatorname{r}}
\def \tp{\operatorname{tp}}
\def \Fix{\operatorname{Fix}}

\newtheorem{theorem}{Theorem}[section]
\newtheorem{lemma}[theorem]{Lemma}

\newtheorem{corollary}[theorem]{Corollary}

\newtheorem{definition}[theorem]{Definition}

\newtheorem{remark}[theorem]{Remark}

\theoremstyle{definition}

\title[Valued Fields Equipped with a Value-preserving Automorphism]{ 
Equivalence of valued fields with a valuation preserving automorphism}
\author{Salih Azgin} 
\address{Department of Mathematics and Statistics, McMaster University 1280 Main Street West
Hamilton, Ontario
Canada  L8S 4K1}
\author{Lou van den Dries}
\address{Department of Mathematics, University of Illinois at Urbana-Champaign, 1409 W. Green Street, Urbana, IL 61801, USA }

\begin{document}

\maketitle

\begin{section}{introduction}

\noindent
Our goal in this paper is to contribute to the model theory of valued fields with a valuation preserving automorphism. The key ideas are due to Scanlon (\cite{scanlon1}, \cite{scanlon2}) and to B\'{e}lair, Macintyre, Scanlon \cite{BMS}. We
obtain the main result in \cite{BMS} under weaker assumptions,
and give a simpler proof. 

\medskip\noindent
Throughout we consider valued fields as 
three-sorted structures $$\ca{K}=(K,\Gamma,\bk;\ v, \pi)$$ 
where $K$ is the underlying field, $\Gamma$ is an ordered abelian group\footnote{The ordering of an ordered abelian group is total, by convention.} 
(the {\em value group\/}), $\bk$ is a field, the surjective map 
$v: K^\times \to \Gamma$ is the valuation, 
with valuation ring
$$  \ca{O}=\ca{O}_v:=\{a \in K:\ v(a) \geq 0\}$$ and maximal ideal 
$\fr{m}_v:=\{a \in K: v(a)>0\}$ of $\ca{O}$, and 
$\pi: \ca{O}\to \bk$ is a surjective ring morphism. Note that then 
$\pi$ induces an isomorphism of fields,
$$  a+\fr{m} \mapsto \pi(a):\ \ca{O}/\fr{m} \to  \bk \qquad (\fr{m}:=\fr{m}_v) $$
and there will be no harm in identifying the residue field 
$\ca{O}/\fr{m}$ with $\bk$ via this isomorphism. Accordingly, we refer 
to $\bk$ as 
the {\em residue field}. To simplify notation we often write $\bar{a}$ 
instead of $\pi(a)$. We call $\ca{K}$ as above
{\em unramified} if either \begin{enumerate}
\item[(i)] $\cha{K}=\cha{\bk}=0$, or
\item[(ii)] $\cha{K}=0$, $\cha{\bk}=p>0$ and $v(p)$ is the smallest 
positive element of $\Gamma$.
\end{enumerate}

\medskip\noindent
Ax \& Kochen and Ershov 
proved the following classical 
result, which we shall refer to as the {\em {\rm AKE}-principle}. (See Kochen \cite{koch}
for a complete account.) 

\medskip\noindent
{\em Let $\ca{K}$ and $\ca{K}'$ be unramified henselian valued fields with
residue fields $\bk$ and $\bk'$ and value groups $\Gamma$ and $\Gamma'$ 
respectively.
Then $\ca{K} \equiv \ca{K}'$ if and only if $\bk \equiv \bk'$, as fields,
and $\Gamma \equiv \Gamma'$, as ordered abelian groups.} 

\medskip\noindent
Thus the elementary theory of an unramified henselian valued
field is determined by the elementary theories of its residue field
and value group. Theorems~\ref{embed11} and ~\ref{embed13}
below are strong analogues of the AKE-principle in the presence of a 
valuation preserving automorphism.

Compared to the standard way of proving the AKE-theorems as in
Kochen~\cite{koch}, the two new tools we need are: replacing a pseudo-Cauchy 
sequence by an equivalent one, in order to rescue pseudocontinuity, and, for 
positive residue characteristic, the use of a certain polynomial 
transformation, the $D$-transform. These two devices were introduced 
in ~\cite{BMS}, but we use them in combination with a simpler 
notion of {\em pc-sequence of $\sigma$-algebraic type\/}.
The main improvement comes from a better understanding of purely residual 
extensions via Lemma~\ref{extresa} below. This allows us to drop a strong 
assumption, 
the {\em Genericity Axiom} of \cite{BMS}, about the 
automorphism induced on the residue field. Other differences 
with \cite{BMS} will
be indicated at various places, but to make our treatment 
selfcontained we include expositions of relevant parts of \cite{BMS}. 
We assume familiarity with valuation theory, including henselization and 
pseudo-convergence. 

\medskip\noindent
This paper is partly based on a chapter of the first author's 
thesis \cite{salih1} with advice from the second author. 

\end{section}

\begin{section}{Preliminaries}

\noindent
Throughout, $\mathbb{N}=\{0,1,2,\dots\}$, and $m,n$ range over $\mathbb{N}$. We let $K^\times=K\setminus \{0\}$ be the multiplicative group of a field $K$. 

\medskip\noindent
{\bf Difference fields.}
A {\em difference field\/} is a field equipped with a distinguished 
automorphism of the field, the {\em difference operator}. A difference field is considered in the obvious way as a structure for the 
language $\{0,1, -, +, \cdot, \sigma\}$ of difference rings, with the unary function 
symbol $\sigma$ to be interpreted in a difference field as its difference operator, which is accordingly also denoted by $\sigma$ (unless specified otherwise). 
Let $K$ be a difference field. The {\em fixed field\/} of $K$ is its subfield
$$\operatorname{Fix}(K):= \{a\in K:\ \sigma(a)=a\}.$$
We let $\sigma^n$ denote the $n^{\rm th}$ iterate of $\sigma$ and 
let $\sigma^{-n}$ denote 
the $n^{\rm th}$ iterate of $\sigma^{-1}$. Let $K \subseteq K'$ be 
an extension of difference fields 
and $a \in K'$. We define $K\langle a \rangle$ to be the smallest difference subfield of $K'$ 
containing $K$ and $a$. The underlying field of $K\langle a \rangle$ is 
$K(\sigma^{i}(a): i \in \mathbb{Z})$.   

\medskip\noindent 
We now introduce difference polynomials in one variable over $K$.
Each polynomial $$f(x_0,\dots,x_n)\in K[x_0,\dots,x_n]$$ 
gives rise to a
difference polynomial $F(x)= f(x, \sigma(x), \dots,\sigma^n(x))$
in the variable $x$ over $K$; we put 
$\deg F := \deg f\in \mathbb{N} \cup\{-\infty\}$ (where $\deg f$ is the 
total degree of $f$), and refer to $F$ as a {\em $\sigma$-polynomial\/} (over $K$). 
If $F$ is not constant (that is, $F\notin K$), let $f(x_0,\dots,x_n)$ 
be as above with least possible $n$ (which determines $f$ uniquely), and put
$$\text{order}(F):= n, \quad 
\text{complexity}(F):= (n,\ \deg_{x_n} f,\ \deg f)\in \mathbb{N}^3.$$
If $F\in K$, $F\ne 0$, then $\text{order}(F):=-\infty$ and  
$\text{complexity}(F):=(-\infty, 0,0)$.
Finally, $\text{order}(0):= -\infty$
and $\text{complexity}(0):=(-\infty, -\infty, -\infty)$. So in all cases 
we have
$\text{complexity}(F)\in (\mathbb{N}\cup \{-\infty\})^3$, and we order
complexities lexicographically.

\medskip\noindent
Let $a$ be an element of a difference field extension of $K$. We say that
$a$ is {\em $\sigma$-transcendental\/} over $K$ if there is no nonzero
$F$ as above with
$F(a)=0,$ and otherwise $a$ is said to be {\em $\sigma$-algebraic\/} over $K$.  
As an example, let $F(x):= \sigma(x)-x$. It has order $1$, and $F(a)=0$
for all $a$ in the prime subfield of $K$, in particular, $F(a)=0$ for
infinitely many $a\in K$ if $K$ has characteristic $0$. If $b$ is also an element in a difference field extension of $K$ and $a$ and $b$ are $\sigma$-transcendental over $K$,
then there is a unique difference field isomorphism $K\langle a \rangle \to K \langle b \rangle$ over $K$ sending $a$ to $b$.

A {\em minimal
$\sigma$-polynomial of $a$ over $K$ \/} is a
nonzero $\sigma$-polynomial $F(x)$ over $K$ such that $F(a)=0$ and $G(a) \neq
0$ for all nonzero $\sigma$-polynomials $G(x)$ over $K$ of lower
complexity than $F(x)$. So $a$ has a minimal $\sigma$-polynomial over $K$ iff
$a$ is $\sigma$-algebraic over $K$. 
{\em Suppose $b$ is also an element
in some difference field extension of $K$, and $a$ and $b$ have a common minimal
$\sigma$-polynomial $F(x)$ over $K$. Is there a difference field
isomorphism $K\langle a \rangle \to K\langle b \rangle$ over $K$ sending $a$ to $b$?}
The answer is not always {\em yes}, but it is {\em yes\/} if $F$ is of degree $1$
in $\sigma^m(x)$ with $F$ of order $m$. Another case in which the answer is {\em yes\/} is treated in Lemma~\ref{extresa} below.

A difference field extension $L$ of $K$
is said to be {\em $\sigma$-algebraic\/} over $K$ if each $c\in L$ is $\sigma$-algebraic
over $K$. For example, if $a$ is $\sigma$-algebraic over $K$, then 
$K\langle a\rangle$ is $\sigma$-algebraic over $K$.

\bigskip\noindent
Let $x_0,\dots,x_n,y_0,\dots,y_n$ be distinct indeterminates, and put
$\mathbf{x}=(x_0,\dots,x_n)$, $\mathbf{y}=(y_0,\ldots,y_n)$. 
For a polynomial $f(\mathbf{x})$ over a field $K$ we have a unique
Taylor expansion in $K[\mathbf{x},\mathbf{y}]$:
$$f(\mathbf{x}+ \mathbf{y})=\sum_{\i}f_{(\i)}(\mathbf{x})\cdot
\mathbf{y}^{\i} ,$$
where the sum is over all $\i = (i_0,\dots,i_n)\in{\mathbb
N}^{n+1}$, each $f_{(\i)}(\mathbf{x})\in K[\mathbf{x}]$, with
$f_{(\i)}=0$ for $|\i|:=i_0+\cdots+i_n > \deg F$, and  
$\mathbf{y}^{\i}:=y_0^{i_0}\cdots y_n^{i_n}$. (Also, for a
tuple $a=(a_0,\dots,a_n)$ with components $a_i$ in any field we put 
$a^{\i}:=a_0^{i_0}\cdots a_n^{i_n}$.)
Thus $\i! f_{(\i)}(\mathbf{x}) =\partial_{\i} f$ 
where $\partial_{\i}$ is the operator $({\partial}/\partial x_0)^{i_0}
\cdots
({\partial}/\partial x_n)^{i_n}$ on $K[\mathbf{x}]$, 
and \linebreak $\i!=i_0!\cdots i_n!$.  
We construe ${\mathbb N}^{n+1}$ as a monoid under $+$ (componentwise addition),  
and let $\leq$ be the (partial)
product ordering on ${\mathbb N}^{n+1}$ induced by the natural order on ${\mathbb N}$.
Define $\left(
\begin{array}{c}
\i\\
\j
\end{array}
\right)$
 as
$\left(
\begin{array}{c}
i_0\\
j_0
\end{array}
\right)\cdots
\left(
\begin{array}{c}
i_n\\
j_n
\end{array}
\right)
\in {\mathbb N}$, when $\j\leq \i$ in ${\mathbb N}^{n+1}$.  Then:

\begin{lemma} 
For $\i,\j \in {\mathbb N}^{n+1}$ we have $(f_{(\i)})_{(\j)}=
\left(
\begin{array}{c}
\i+\j\\
\i
\end{array}
\right) f_{(\i+ \j)}$.
\end{lemma} 

\noindent
In particular, $f_{(\i)}=f$ for $|\i|=0$,
and if
$|\i|=1$ with $i_k=1$, then
$f_{(\i)}= \frac{\partial f}{\partial x_k}.$  Also, 
$\deg f_{(\i)}< \deg f$ if $|\i|\geq 1$ and $f\ne 0$.

\medskip\noindent
Let $K$ be a difference field, and $x$ an indeterminate. 
When $n$ is clear from context we set 
$\boldsymbol{\sigma}(x)=(x, \sigma(x),\dots,\sigma^n(x))$, and also
$\boldsymbol{\sigma}(a)=(a, \sigma(a),\dots,\sigma^n(a))$ for $a\in K$. 
Then for $f\in K[x_0,\dots,x_n]$ as above and 
$F(x)= f(\bsigma(x))$ we have the 
following identity in the ring of difference polynomials in the distinct
indeterminates $x$ and $y$ over $K$:
\begin{eqnarray*}
F(x+y)&=&f(\bsigma(x+y))
=f(\bsigma(x)+\bsigma(y))\\
&=&\sum_\i f_{(\i)}(\bsigma(x))\cdot
\bsigma(y)^{\i} = \sum_\i F_{(\i)}(x)\cdot
\bsigma(y)^{\i},
\end{eqnarray*}
where $F_{(\i)}(x):=f_{(\i)}(\bsigma(x)).$

\bigskip\noindent
{\bf Valued fields.}
We consider valued fields as three-sorted 
structures $$\ca{K}=(K,\Gamma,\bk;v, \pi)$$ as explained in the introduction.
The three sorts are referred to as the {\em field sort\/} with variables
ranging over $K$, the {\em value group sort\/} with variables ranging over
$\Gamma$, and the {\em residue sort\/} with variables ranging over $\bk$. We 
say that $\ca{K}$ is of {\em equal characteristic $0$} if 
$\cha(K)=\cha(\bk)=0$. If $\cha(K)=0$ and $\cha(\bk)=p>0$, we say that 
$\ca{K}$ is of {\em mixed characteristic}.  

In dealing with a valued field $\ca{K}$ as above we also let
$v$ denote the valuation of any valued field
extension of $\ca{K}$ that gets mentioned,
unless we indicate otherwise, and any subfield $E$ of $K$ is construed 
as a {\em valued\/} subfield of $\ca{K}$ in the obvious way. 

\medskip\noindent
A valued field extension $\ca{K}'$ of a valued field $\ca{K}$ is said to be {\em immediate} if the residue field and the value group of $\ca{K}'$ are
the same as those of $\ca{K}$. A valued field is {\em maximal} if it has 
no proper immediate valued field extension and is {\em algebraically maximal} if it has no proper immediate algebraic valued field extension.

\medskip\noindent
A key notion in the study of immediate extensions of valued fields
is that of pseudo-cauchy sequence. First, a {\em well-indexed sequence}
is a sequence $\{a_\rho\}$ indexed by the elements $\rho$ of some
nonempty well-ordered set without largest element; in this connection
``eventually'' means ``for all sufficiently large $\rho$''.

Let $\ca{K}$ be a valued field.
A {\em pseudo-Cauchy sequence} (henceforth {\em pc-sequence}) in 
$\ca{K}$ is a well-indexed sequence
$\{a_\rho\}$ in $K$ such that for some index
$\rho_0 $, $$\rho'' > \rho' > \rho \ge \rho_0\ \Longrightarrow\ 
v(a_{\rho''}-a_{\rho'}) > v(a_{\rho'}-a_\rho).$$
In particular, a pc-sequence in $\ca{K}$ cannot be eventually constant.
For a well-indexed sequence
$\{a_\rho\}$ in $\ca{K}$ and $a$ in some valued field extension of $\ca{K}$ 
we say that 
$\{a_\rho\}$ {\em pseudoconverges\/} to $a$, or $a$ is a 
{\em pseudolimit\/} of 
$\{a_\rho\}$ (notation: $a_\rho \leadsto a$) if the sequence 
$\{v(a-a_\rho)\}$ is eventually strictly increasing; note that then
$\{a_\rho\}$ is a pc-sequence in $\ca{K}$.  

Let $\{a_\rho\}$ be a pc-sequence in $\ca{K}$, pick $\rho_0$ 
as above, and put 
$$\gamma_\rho :=v(a_{\rho'}-a_\rho)$$ for
$\rho'>\rho \ge \rho_0$; this depends only on $\rho$ as the notation
suggests. Then $\{\gamma_\rho\}_{\rho \ge \rho_0}$ is strictly increasing.
The {\em width} of $\{a_\rho\}$ is the set
$$\{\gamma \in
\Gamma\cup\{\infty\}:\gamma>\gamma_\rho\ \mbox{for all}\ \rho \ge \rho_0\}.$$ 
Its significance is that if $a,b\in K$ and $a_\rho \leadsto a$, then 
$a_\rho \leadsto b$ if and only if $v(a-b)$ is in the width of $\{a_\rho\}$.

An old and useful observation by Macintyre is that if $\{a_\rho\}$ is a pc-sequence 
in an expansion of a valued field (for example, in a valued difference field), 
then $\{a_\rho\}$ has a pseudolimit in an elementary extension of that
expansion.

\medskip\noindent
The following easy lemma will be useful in dealing with pc-sequences.

\begin{lemma}\label{kapla} Let $\Gamma$ be an ordered abelian group, 
$A$ a subset of
$\Gamma$, and $\{\gamma_\rho\}$ a well-indexed strictly increasing sequence 
in $A$. Let $f_1,\dots,f_n: A \to \Gamma$ be such that
for all distinct $i,j\in \{1,\dots,n\}$ the function
$f_i - f_j$ is either strictly increasing or strictly decreasing.
Then there is a unique enumeration $i_1,\dots,i_n$ of $\{1,\dots,n\}$
such that 
$$f_{i_1}(\gamma_\rho) < \dots < f_{i_n}(\gamma_\rho), \quad \text{eventually}.$$
For this enumeration and $\delta\in \Gamma$ such that
$\{\gamma\in \Gamma:\ 0 < \gamma < \delta\}$ is finite, if
$1\le \mu < \nu \le n$, then
$f_{i_\nu}(\gamma_\rho) - f_{i_\mu}(\gamma_\rho) > \delta,$ eventually.
\end{lemma}

\noindent
For linear functions on $\Gamma$ this was used by Kaplansky~\cite{kaplansky} 
in his work on immediate extensions of valued fields. The last part of the 
lemma is needed in dealing with finitely ramified valued fields of mixed
characteristic. As in \cite{BMS} we call the valued field $\ca{K}$
{\em finitely ramified\/} if the following two conditions are satisfied:
\begin{enumerate}
\item[(i)] $K$ has characteristic $0$;
\item[(ii)] $\{\gamma\in \Gamma:\ 0<\gamma < v(p)\}$ is finite if
$\bk$ has characteristic $p>0$.
\end{enumerate} 
In particular, $\ca{K}$ is finitely ramified if
$\ca{K}$ is unramified as defined in the introduction.

\medskip\noindent
Let $\ca{K}=(K, \Gamma, \bk; v, \pi)$ be a valued field. A {\em cross-section\/}
on $\ca{K}$ is a group morphism $s: \Gamma \to K^\times$ such that $v(s\gamma)=\gamma$ for all $\gamma\in \Gamma$. The following is well-known:

\begin{lemma}\label{xs1} If $\ca{K}$ is $\aleph_1$-saturated, then there is 
a cross-section on $\ca{K}$. In particular, there is a
cross-section on some elementary extension of $\ca{K}$.
\end{lemma}
\begin{proof} With $U(\mathcal{O})$ the multiplicative group of units of 
$\mathcal{O}$, the inclusion $U(\mathcal{O}) \to K^\times$ and 
$v: K^\times \to \Gamma$ yield the exact sequence of abelian groups
$$ 1\ \to\ U(\mathcal{O})\ \to\  K^\times\ \to\  \Gamma\ \to\  0.$$
Suppose $\ca{K}$ is $\aleph_1$-saturated. Then the group 
$U(\mathcal{O})$ is $\aleph_1$-saturated, hence pure-injective by 
\cite{cherlin}, p. 171. It is also pure in $K^\times$ since $\Gamma$ is 
torsion-free, and thus the above exact sequence splits. 
\end{proof} 

\noindent
In the proof of Theorem~\ref{embed11a} we need the following variant:

\begin{lemma}\label{xs2}
Let $\ca{K}$ be $\aleph_1$-saturated, let $\ca{E}=(E, \Gamma_E, \dots)$ 
be an $\aleph_1$-saturated valued subfield of $\ca{K}$ such that 
$\Gamma_E$ is pure in $\Gamma$, and let $s_E$ be a cross-section on $\ca{E}$. 
Then $s_E$ extends to a cross-section on $\ca{K}$.
\end{lemma}
\begin{proof} By Lemma~\ref{xs1} we have a cross-section $s$ on $\ca{K}$. Now $\Gamma_E$ is pure-injective, and pure in $\Gamma$, so we have an internal direct sum decomposition $\Gamma=\Gamma_E \oplus \Delta$ with $\Delta$ a subgroup of $\Gamma$. This gives a cross-section on $\ca{K}$ that concides with 
$s_E$ on $\Gamma_E$ and with $s$ on $\Delta$.
\end{proof}

\noindent
An {\em angular component map\/} on $\ca{K}$ is a 
multiplicative group morphism $\ac \colon K^{\times} \to \bk^{\times}$
such that $\ac(a)=\pi(a)$ whenever $v(a)=0$; we extend it to $\ac \colon K \to \bk$ by setting $\ac(0)=0$ and also refer to this extension as an angular 
component map on $\ca{K}$. A cross-section $s$ on $\ca{K}$ yields an angular 
component map $\ac$ on $\ca{K}$ by setting $\ac(x)= \pi\big(x/s(v(x))\big)$ for 
$x\in K^\times$. Thus Lemma~\ref{xs1} goes through with angular component maps
instead of cross-sections.  



\bigskip\noindent
{\bf Valued difference fields.} A {\em valued difference field\/} is a valued field $\ca{K}$ as above 
where $K$ is not just a field, but a difference field whose
difference operator $\sigma$ satisfies
$\sigma(\ca{O})= \ca{O}$. It follows that
$\sigma$ induces an automorphism of the residue field:
$$ \pi(a) \mapsto \pi(\sigma(a)):\ \bk \to \bk, \quad a \in \ca{O}.$$  
We denote this automorphism by $\bar{\sigma}$, and $\bk$ equipped with
$\bar{\sigma}$ is called the {\em residue difference field of $\ca{K}$}. (Likewise, $\sigma$ induces an automorphism of the value group $\Gamma$; at a later stage
we restrict attention to $\ca{K}$ where $\sigma$ induces the identity on $\Gamma$.) 

\medskip\noindent
Let $\ca{K}$ be a valued difference field as above. The difference operator $\sigma$ of $K$ is also referred to as the {\em difference operator of $\ca{K}$}. By an {\em extension\/} of $\ca{K}$ we shall mean a valued difference field $\ca{K}'=(K',\dots)$ that extends $\ca{K}$ as a valued field and whose difference operator extends the difference operator of $\ca{K}'$. In this situation we also say that $\ca{K}$ is a {\em valued difference subfield of\/} $\ca{K}'$, and we indicate this by $\ca{K} \le \ca{K}'$. Such an extension is called {\em immediate\/} 
if it is immediate as an extension of valued fields.
In dealing with a valued difference field $\ca{K}$ as above $v$ also denotes the valuation of any extension of $\ca{K}$ that 
gets mentioned (unless specified otherwise), and any difference 
subfield $E$ of $K$ is construed 
as a valued difference subfield of $\ca{K}$ in the obvious way. The 
residue field of the valued subfield $\text{Fix}(K)$ of $\ca{K}$ is clearly a
subfield of $\text{Fix}(\bk)$. 

Let $\ca{K}^h= (K^h, \Gamma, \bk;\dots)$ be the henselization of the 
underlying valued field of $\ca{K}$. By the universal property of ``henselization''
the operator $\sigma$ extends uniquely to an automorphism $\sigma^h$ of the field $K^h$ such that $\ca{K}^h$ with $\sigma^h$ is a valued difference field. Accordingly we shall view $\ca{K}^h$ as a valued difference field, making it thereby an immediate extension of the valued difference field $\ca{K}$.

Given an extension  $\ca{K} \leq \ca{K}'$ of valued difference fields and 
$a \in K'$, we define $\ca{K} \langle a \rangle$ to be the 
smallest valued difference subfield of $\ca{K}'$ extending $\ca{K}$ and 
containing $a$ in its underlying difference field; thus
the underlying difference field of $\ca{K} \langle a \rangle$ 
is $K \langle a \rangle$.

\bigskip\noindent
{\bf Two lemmas.} Suppose
$\ca{K}=(K,\Gamma,\bk;v, \pi)$ and $\ca{K}'=(K', \Gamma', \bk'; v', \pi')$ are valued difference fields, put $\ca{O}:= \ca{O}_v,\ \ca{O}':= \ca{O}_{v'}$, and 
let $\sigma$ denote both the difference operator of $\ca{K}$ and of $\ca{K}'$.  Let $\ca{E}=(E, \Gamma_E, \bk_E; \dots)$ be a valued difference subfield of both $\ca{K}$ and $\ca{K}'$, that is,
$\ca{E} \le \ca{K}$ and $\ca{E}\le \ca{K}'$.
The next lemma is rather obvious, but Lemma~\ref{extresa} is more subtle and our later use of it is one of the reasons that we never need the 
Genericity Axiom of \cite{BMS}.

\begin{lemma}\label{extrest}  Let $a \in \ca{O}$ and assume
$\alpha=\bar{a}$ is $\bar{\sigma}$-transcendental over $\bk_E$. 
Then \begin{enumerate}
\item[(i)] $v(P(a))= \min\limits_{\l} \{v(b_{\l})\}$  for each
$\sigma$-polynomial $P(x)=\sum b_{\l} \boldsymbol{\sigma}^{\l}(x)$ over
$E$;
\item[(ii)]  $v(E \langle a \rangle^{\times})=v(E^{\times})=\Gamma_E$, and
 $\ca{E} \langle a \rangle$ has residue field $\bk_E \langle
\alpha \rangle$;
\item[(iii)] if $b\in \ca{O}'$ is such that $\beta=\bar{b}$ is $\bar{\sigma}$-transcendental over $\bk_E$, then there is
a valued difference field isomorphism $\ca{E}\langle a \rangle \to  \ca{E} \langle b \rangle$ over $\ca{E}$ sending $a$ to $b$. 
\end{enumerate} 
\end{lemma}

\begin{proof} Let
$P(x)=\sum b_{\l} \boldsymbol{\sigma}^{l}(x)$ be a nonzero
$\sigma$-polynomial over $E$. Then $P(x)=cQ(x)$ where $c
\in E^{\times}$ and $v(c)=\min\limits_{\l} \{v(b_{\l})\}$, and $Q(x)$ is a $\sigma$- polynomial over the valuation
ring of $E$ with some coefficient equal to $1$. Since
$\alpha=\bar{a}$ is $\bar{\sigma}$-transcendental over $\bk_E$,
$\bar{Q}(\bar{a}) \neq 0$. Therefore $v(Q(a))=0$, and thus
$$v(P(a))= v(c)=\min\limits_{\l} \{v(b_{\l})\}.$$ 
It follows that $v(E \langle a
\rangle^{\times})=v(E^{\times})$. A similar argument
shows that $E \langle a \rangle$ has 
residue field $\bk_E \langle \alpha \rangle$. 
It also follows from (i) that $a$ is $\sigma$-transcendental over $E$, and (iii) is an easy consequence of this fact and of (i).
\end{proof}

\noindent
Recall that in the beginning of this section we defined the 
complexity of a difference polynomial over a difference field.

\begin{lemma}\label{extresa}\footnote{We thank Martin Hils
for pointing out a serious error in the proof of a related lemma in \cite{salih1}.}
Assume $\operatorname{char}(\bk)=0$, and let $G(x)$ be a nonconstant
$\sigma$-polynomial over the valuation ring of $E$ whose reduction
$\bar{G}(x)$ has the same complexity as $G(x)$. Let $a\in \ca{O}$, 
$b\in \ca{O}'$, and assume that
$G(a)=0$, $G(b)=0$, and that $\bar{G}(x)$ is a minimal $\bar{\sigma}$-polynomial of $\alpha:=\bar{a}$ and of $\beta:= \bar{b}$ over $\bk_E$. Then
\begin{enumerate}
\item[(i)] $\ca{E} \langle a \rangle$
has value group $v(E^{\times})=\Gamma_E$ and residue field 
$\bk_E \langle \alpha \rangle$;
\item[(ii)] if there is a difference field isomorphism $\bk_E\langle \alpha \rangle \to \bk_E \langle \beta \rangle$ over $\bk_E$ sending $\alpha$ to $\beta$, then there is
a valued difference field isomorphism $\ca{E}\langle a \rangle \to  \ca{E} \langle b \rangle$ over $\ca{E}$ sending $a$ to $b$. 
\end{enumerate} 
\end{lemma}
\begin{proof} To simplify notation we set, for $k\in \mathbb{Z}$,
$$a_k:= \sigma^k(a),\quad \alpha_k:= \bar{\sigma}^k(\alpha),\quad b_k:= \sigma^k(b),\quad \beta_k:= \bar{\sigma}^k(\beta).$$ 
As in the proof of Lemma~\ref{extrest} one shows that
if $P(x)=\sum b_{\l} \boldsymbol{\sigma}(x)^{\l}$ is a
$\sigma$-polynomial over $E$ of lower complexity than $G(x)$, then
$v(P(a))=\min\limits_{\l}\{v(b_{\l})\}$. It is also clear that $G$ is a minimal $\sigma$-polynomial of $a$ over $E$. Let $G$ have order $m$ and degree 
$d>0$ with respect to $\sigma^m(x)$, so  
$$G(x)= P_0(x) + P_1(x)\sigma^m(x) + \cdots + P_d(x)\sigma^m(x)^d $$ 
where
$P_0, \dots, P_d$ are $\sigma$-polynomials over the valuation ring of $E$ of order less than $m$, with $P_d\ne 0$. 
Then the valued subfield $E_{m-1}:=E(a_0,\dots, a_{m-1})$ of $\ca{K}$ has transcendence basis $a_0,\dots, a_{m-1}$ over $E$, the residue field of $E_{m-1}$ is
$\bk_E\big(\alpha_0,\dots,\alpha_{m-1})$ with transcendence basis 
$\alpha_0,\dots,\alpha_{m-1}$ over $\bk_E$, and the value group of $E_{m-1}$ is $\Gamma_E$. 
 Also,  
$v(P_d(a))=0$ and $v(P_i(a))\ge 0$ for $i=0,\dots, d-1$, and 
$$g(T):=T^d + p_{d-1}T^{d-1} + \dots + p_0, \quad \text{with } p_i:= P_i(a)/P_d(a) \text{ for }i=0,\dots, d-1,$$ 
is the minimum polynomial of $a_m$
over $E_{m-1}$ and has its coefficients in the valuation ring of $E_{m-1}$, 
and the reduction $\bar{g}(T)$ of $g(T)$ is the minimum polynomial of $\alpha_m$ over $\bk_E(\alpha_0,\dots, \alpha_{m-1})$.
For the rest of the proof we assume without loss that $\ca{K}$ and $\ca{K}'$ are henselian as valued fields. 

\medskip\noindent
For $n\ge m$ we set
 $$E_n:= E(a_0, \dots, a_n), \quad \text{ a valued subfield of } \ca{K},$$
we let $E_n^h$ be the henselization of $E_n$ in $\ca{K}$, and let $E_{m-1}^h$ be the henselization of $E_{m-1}$ in $\ca{K}$. For $n\ge m$, let $g_n(T)$ be the minimum polynomial of 
$a_n$ over $E_{n-1}^h$.

\medskip\noindent  
{\bf Claim 1}: for $n\ge m$ the polynomial $g_n$ has its coefficients in the valuation ring of $E_{n-1}^h$, the residue field of $E_n$ is $\bk_E(\alpha_0,\dots, \alpha_n)$, the value group of $E_n$ is $\Gamma_E$, 
the reduction $\bar{g}_n$ of $g_n$ is the minimum polynomial
of $\alpha_n$ over $\bk_E(\alpha_0,\dots, \alpha_{n-1})$.

\medskip\noindent
Claim 1 holds for $n=m$: $\bar{g}(T)$ is
the minimum polynomial
of $\alpha_m$ over the residue field $\bk_E(\alpha_0,\dots, \alpha_{m-1})$ of $E_{m-1}$, and so the monic polynomial $g(T)$ is
necessarily the minimum polynomial of $a_m$ over $E_{m-1}^h$, that is,
$g_m=g$. Assume inductively that the claim holds for a certain $n\ge m$.
By applying $\sigma$ to the coefficients of $g_n(T)$ we obtain a monic polynomial 
$g_n^{\sigma}(T)$ over the valuation ring of 
$E_{n}^h$ with $a_{n+1}$ as a zero. Thus $g_{n+1}(T)$ is a monic irreducible factor of $g_n^{\sigma}(T)$ in $E_{n}^h[T]$, and has therefore coefficients in the valuation ring of $E_{n}^h$. Its reduction $\bar{g}_{n+1}$ divides
the reduction of $g_n^{\sigma}$ (in the polynomial ring $\bk_E( \alpha_0, \dots, \alpha_n)[T]$) and so 
$\alpha_{n+1}$, being a simple zero of this last reduction, is a simple zero of $\bar{g}_{n+1}$. It only remains to show that $\bar{g}_{n+1}$ is irreducible in $\bk_E( \alpha_0, \dots, \alpha_n)[T]$. Suppose it is not.
Then $\bar{g}_{n+1}(T)=\phi(T)\psi(T)$ where $\phi,\psi\in \bk_E( \alpha_0, \dots, \alpha_n)[T]$ are monic of degree $\ge 1$, with $\phi$ irreducible in this polynomial ring and $\phi(\alpha_{n+1})=0$. Hence $\phi$ and
$\psi$ are coprime. Then the factorization $\bar{g}_{n+1}=\phi\psi$ can be lifted to
a nontrivial factorization of $g_{n+1}$ in $E_{n}^h[T]$, a contradiction.
Claim 1 is established.

It follows that $E(a_k:\ k\in \mathbb{N})$ has residue field 
$\bk_E(\alpha_k:\ k\in \mathbb{N})$ and value group $\Gamma_E$. Applying the valued field automorphism
$\sigma^{-n}$ yields that the valued subfield $E(a_{k-n}:\ k\in \mathbb{N})$ of $\ca{K}$ has residue field 
$\bk_E(\alpha_{k-n}:\  k\in \mathbb{N})$ and value group $\Gamma_E$.
Hence $E\langle a \rangle$ has residue field $\bk_E\langle \alpha \rangle$ and value group $\Gamma_E$. We have proved (i).

\medskip\noindent
To prove (ii), let $\iota: \bk_ E\langle \alpha\rangle \to \bk_E\langle \beta\rangle$
be a difference field isomorphism 
over $\bk_E$ sending $\alpha$ to $\beta$. Let $F_{m-1}:= E\big(b_0,\dots, b_{m-1})$, a valued
subfield of $\ca{K}'$. Then
$$h(T):=T^d + q_{d-1}T^{d-1} + \dots + q_0, \quad \text{with } q_i:= P_i(b)/P_d(b) \text{ for }i=0,\dots, d-1,$$ 
is the minimum polynomial of $b_m$
over $F_{m-1}$ and has its coefficients in the valuation ring of $F_{m-1}$, 
and the reduction $\bar{h}(T)$ of $h(T)$ is the minimum polynomial of $\beta_m$ over $\bk_E(\beta_0,\dots, \beta_{m-1})$.
Now $\bar{g}$ and $\bar{h}$ correspond under $\iota$. For $n\ge m$, let 
$$F_n:= E(b_0,\dots, b_n), \quad \text{ a valued subfield of }\ca{K}',$$
and let $F_n^h$ be the henselization of $F_n$ in $\ca{K}'$.
For $n\ge m$, let $h_n(T)$ be the minimum polynomial of $b_n$ over
$F_{n-1}^h$, so $h_n$ has its coefficients in the valuation ring of $F_{n-1}^h$, the residue field of $F_n$ is $\bk_E(\beta_0,\dots, \beta_n)$, the value group of $F_n$ is $\Gamma_E$, the reduction $\bar{h}_n$ of $h_n$ is the minimum polynomial of $\beta_n$ over $\bk_E(\beta_0,\dots, \beta_{n-1})$. It follows that $\bar{g}_n$ and $\bar{h}_n$
correspond under $\iota$, for each $n\ge m$. 

\medskip\noindent
{\bf Claim 2}: for $n\ge m$ there is a (unique) valued field isomorphism 
$i_n\ :\ E_n \to F_n$ over $E$ sending $a_k$ to $b_k$ for $k=0,\dots,n$. 

\medskip\noindent
From the remarks at the beginning of the proof it is clear that we have a unique valued field isomorphism $i_{m-1}\ :\ E_{m-1} \to F_{m-1}$ over $E$ sending $a_k$ to $b_k$ for $k=0,\dots,m-1$. It follows that the minimum polynomials $g$ and $h$ correspond under $i_{m-1}$, and so we have a field isomorphism
$E_m \to F_m$ extending $i_{m-1}$ and sending $a_m$ to $b_m$. This is a {\em valued\/} field isomorphism since the residue field  
$\bk_E(\alpha_0, \dots, \alpha_m)$ of $E_m$ has the same degree
over the residue field $\bk_E(\alpha_0, \dots,\alpha_{m-1})$ of $E_{m-1}$ as $E_m$ has over $E_{m-1}$, and likewise with $F_m$ and $F_{m-1}$. This proves Claim 2 for $n=m$.
Assume the claim holds for a certain $n\ge m$. Then $g_n$ and $h_n$ correspond under
$i_{n-1}$, and so $g_n^\sigma$ and $h_n^\sigma$ correspond under $i_n$. 
From the unique lifting properties of henselian local rings it follows that $g_{n+1}$ is the unique monic polynomial in $E_n^h[T]$ that divides $g_n^\sigma$, has its coefficients in the valuation ring of $E_n^h$,
and whose reduction is $\bar{g}_{n+1}$; likewise with $h_{n+1}$. 
Therefore $g_{n+1}$ and $h_{n+1}$ correspond under $i_n$, and so 
we have a field isomorphism 
$E_n^h(a_{n+1}) \to F_n^h(b_{n+1})$ that extends $i_n$ and sends
$a_{n+1}$ to $b_{n+1}$. Arguing as in the case $n=m$ we see that this
field isomorphism is a valued field isomorphism; its restriction to $E_{n+1}$ is
the desired $i_{n+1}$. This proves Claim 2, and then it is easy to finish the 
proof of (ii).
\end{proof}

\noindent
Lemma~\ref{extresa} and its proof go through if we replace the assumption $\cha{\bk}=0$ by its consequence that $\alpha_m$ is a simple zero of its minimum polynomial over $\bk_E(\alpha_0, \dots, \alpha_{m-1})$,
where $m$ is the order of $G$ as in the proof. (Just add to Claim 1 in the proof that $\alpha_n$ is a simple zero of $\bar{g}_n$, for all $n\ge m$.)

\bigskip\noindent
{\bf Hahn difference fields and Witt difference fields.} 
Let $\bk$ be a field and $\Gamma$ an ordered abelian group. This gives the Hahn field
$\bk((t^{\Gamma}))$ whose elements are the formal sums
$a=\sum_{\gamma \in \Gamma} a_\gamma t^{\gamma}$  
with $a_\gamma \in \bk$ for all $\gamma$, with well-ordered {\em support\/}
$\{\gamma:\ a_\gamma \neq 0\} \subseteq \Gamma$. With $a$ as above, we define the valuation $v: \bk((t^{\Gamma}))^{\times} \to \Gamma$ by
$v(a):=\min \{\gamma: a_\gamma \neq 0\}$,
and the surjective ring morphism  
$\pi:\ \ca{O}_v \to \bk$ by $\pi(a):=a_0$.
In this way we obtain the (maximal) valued field
$\ca{K}=(\bk((t^\Gamma)), \Gamma, \bk;v,\pi)$ to which we also just refer to
as the {\em Hahn field} $\bk((t^\Gamma))$.

\medskip\noindent
Let the field $\bk$ also be equipped with an
automorphism $\bar{\sigma}$. Then
$$\sum_{\gamma} a_\gamma t^{\gamma} \mapsto 
\sum_{\gamma} \bar{\sigma}(a_\gamma) t^\gamma $$
is an automorphism, to be denoted by $\sigma$, 
of the field $\bk((t^\Gamma))$, with $\sigma(\ca{O}_v)=\ca{O}_v$.
We consider the three-sorted structure 
$(\bk((t^\Gamma)), \Gamma, \bk;\ v,\pi)$, with the field $\bk((t^\Gamma))$
equipped with the
automorphism $\sigma$ as above, as a valued difference field,
and also refer to it as the {\em Hahn difference field} $\bk((t^\Gamma))$.
Thus $\text{Fix}\big(\bk((t^\Gamma))\big)=\text{Fix}(\bk)((t^\Gamma))$.

\bigskip\noindent
Let now $\bk$ be a perfect field of characteristic $p>0$. Then we have the ring $\operatorname{W}[\bk]$ of 
Witt vectors over $\bk$; it is a complete discrete valuation ring whose elements are the infinite sequences $(a_0, a_1, a_2,\dots)$ with all $a_n\in \bk$; see for example \cite{serre} for how addition and multiplication are defined. The Frobenius
automorphism $x \mapsto x^p$ of $\bk$ induces the ring automorphism
$$(a_0,a_1,a_2,\dots) \mapsto (a_0^p, a_1^p, a_2^p,\dots)$$  of $\operatorname{W}[\bk]$. This automorphism of $\operatorname{W}[\bk]$ extends to a field automorphism, the {\em Witt frobenius}, of
the fraction field $\operatorname{W}(\bk)$ of $\operatorname{W}[\bk]$. We consider 
$\operatorname{W}(\bk)$ as a valued difference field by taking the Witt frobenius as
its difference operator, by taking the valuation $v$ to be the unique one with
valuation ring $\operatorname{W}[\bk]$, value group $\mathbb{Z}$ and $v(p)=1$,
and by letting
$\pi: \operatorname{W}[\bk] \to \bk$ be the canonical map 
$$(a_0, a_1, a_2,\dots) \mapsto a_0.$$ 
We refer to this valued difference field as the {\em Witt difference field \/} $\operatorname{W}(\bk)$. For any perfect subfield $\bk'$ of $\bk$ we consider
$\operatorname{W}(\bk')$ as a valued difference subfield of
$\operatorname{W}(\bk)$ in the obvious way. In particular, with $\mathbb{F}_p$
the prime field of $\bk$, we have 
$\text{Fix}\big(\operatorname{W}(\bk)\big)= \operatorname{W}(\mathbb{F}_p)$, and the latter is identified with the valued field $\mathbb{Q}_p$ of 
$p$-adic numbers in the usual way. In the last section the functorial nature of
$\operatorname{W}$ plays a role:  any field embedding 
$\iota: \bk \to \bk'$ into a perfect field $\bk'$ induces the ring embedding
$$\operatorname{W}[\iota]\colon \operatorname{W}[\bk] \to \operatorname{W}[\bk'], \quad
(a_0, a_1, a_2,\dots) \mapsto (\iota a_0, \iota a_1, \iota a_2,\dots).$$

\bigskip\noindent
{\bf Two axioms.}
Let $\ca{K}$ be a valued difference field, and consider the following two conditions on $\ca{K}$. The first one says that $\sigma$ preserves the valuation $v$.

\medskip\noindent
{\bf Axiom 1.} For all $a \in K^{\times}$, $v(\sigma(a))=v(a)$.  

\medskip\noindent
{\bf Axiom 2.} For all $\gamma \in \Gamma$ there is 
$a \in \operatorname{Fix}(K)$ such that
$v(a)=\gamma$.

\medskip\noindent
It is easy to see that Axiom 2 implies Axiom 1. If $\Gamma$ is an ordered abelian group and $\bk$ a difference field, then the Hahn difference field $\bk((t^\Gamma))$ satisfies Axiom 2. If $\bk$ is a perfect field of characteristic $p>0$, then the  
Witt difference field $\operatorname{W}(\bk)$ satisfies Axiom 2. If $\ca{K}$ satisfies Axiom 1, so does any valued difference subfield of $\ca{K}$, and any extension of $\ca{K}$ with
the same value group. If $\ca{K}$ satisfies Axiom 2, so does any extension with the
same value group. 

\medskip\noindent
{\em From now on we assume that all our valued difference fields 
satisfy Axiom $1$}. By this convention, whenever we refer to an extension 
of a valued difference field, this extension is also assumed to satisfy Axiom 1.

\end{section}

\begin{section}{Pseudoconvergence and $\sigma$-polynomials}

\medskip\noindent
If $\{a_\rho\}$ is a pc-sequence in a valued field $K$ and
$a_\rho \leadsto a$ with $a\in K$, then for an ordinary nonconstant 
polynomial $f(x) \in K[x]$ we have $f(a_\rho) \leadsto f(a)$, see \cite{kaplansky}.
This fails in general for nonconstant $\sigma$-polynomials over valued difference fields. We do, however, have a variant of
this pseudo-continuity for $\sigma$-polynomials using
{\em equivalent pc-sequences}. This is a key device from \cite{BMS},
and we follow its treatment, but with some differences. 

\begin{definition}\label{equiv.pc} Two pc-sequences
$\{a_\rho\},\{b_\rho\}$
in a valued field are equivalent if for all $a$ in all 
valued field extensions,
$a_\rho \leadsto a
\Leftrightarrow b_\rho \leadsto a.$
\end{definition}

\noindent
This is an equivalence relation on the set of pc-sequences with given index set and
in a given valued field, and: 

\begin{lemma}
 Two pc-sequences $\{a_\rho\}$ and $\{b_\rho\}$ in a valued field 
are equivalent if and only if they
 have the same width and a
 common pseudolimit in some valued field extension.
\end {lemma}

\bigskip\noindent
{\bf The Basic Calculation.}
Let $\ca{K}$ be a valued difference field satisfying Axiom 2, and
let $\{a_\rho\}$ be a pc-sequence from $K$ such that 
$a_\rho \leadsto a$ with $a$ in an extension of $\ca{K}$. Put $\gamma_\rho:=v(a_\rho-a)$; then $\{\gamma_\rho\}$ is eventually
strictly increasing. Let
$G$ be a nonconstant $\sigma$-polynomial over $K$ of order $\le n$. Under an additional assumption on $\bk$ we shall 
construct a pc-sequence $\{b_\rho\}$ from $K$  
equivalent to $\{a_\rho\}$ such that $G(b_\rho) \leadsto G(a)$. 
We first choose $\theta_\rho\in \operatorname{Fix}(K)$ such that 
$v(\theta_\rho)= \gamma_\rho;$ this is possible by Axiom 2.
We set $b_\rho:=a_\rho+\mu_\rho\theta_\rho$ and for now we demand only that $\mu_\rho\in K$ and $v(\mu_\rho)=0.$
Define $d_\rho$ by $a_\rho-a=\theta_\rho d_\rho,$ so $v(d_\rho)=0$. Then
\begin{eqnarray*}
b_\rho-a&=&b_\rho-a_\rho+a_\rho-a\\
&=&\theta_\rho(\mu_\rho+d_\rho).
\end{eqnarray*}
We now impose also $v(\mu_\rho+d_\rho)=0$. This ensures that 
$b_\rho\leadsto a$, and that $\{a_\rho\}$ and $\{b_\rho\}$
have the same width, so they are equivalent.  Note that $d_\rho$ 
depends only on our choice of $\theta_\rho$ (not on $\mu_\rho$),
and $d_\rho$ won't normally be in $K$.
Now
\begin{eqnarray*}
G(b_\rho)-G(a)&=&
\sum\limits_{|\l|\geq 1}G_{(\l)}(a)\cdot\boldsymbol{\sigma}
(b_\rho-a)^{\l}\\
&=&\sum\limits_{m\geq 1}\sum\limits_{|\l|=m}G_{(\l)}(a)\cdot
\boldsymbol{\sigma}(b_\rho-a)^{\l}\\
&=&\sum\limits_{m\geq 1}\sum\limits_{|\l|=m}G_{(\l)}(a)\cdot
\boldsymbol{\sigma}\big(\theta_\rho(\mu_\rho+d_\rho)\big)^{\l}\\
&=&\sum\limits_{m\geq 1}\sum\limits_{|\l|=m}G_{(\l)}(a)\cdot
\boldsymbol{\sigma}
(\theta_\rho)^{\l}\cdot\boldsymbol{\sigma}(\mu_\rho+d_\rho)^{\l}\\
&=&\sum\limits_{m\geq 1}\theta_\rho^m\cdot G_m(\mu_\rho+d_\rho)
\end{eqnarray*}
where $G_m$ is the $\sigma$-polynomial over $K\langle a\rangle$ given by 
$$G_m(x)=\sum\limits_{|\l|=m}G_{(\l)}(a)\cdot\boldsymbol{\sigma}(x)^{\l}.$$
Since $G\notin K$, there is an $m\ge 1$ such that $G_m\ne 0$.
For such $m$, pick $\l=\l(m)$ with $|\l|=m$ for which
$v\big(G_{(\l)}(a)\big)$ is minimal, so $G_m(x)=
G_{(\l)}(a)\cdot g_m\big(\boldsymbol{\sigma}(x)\big)$
where $g_m(x_0,\dots,x_n)$ has its coefficients 
in the valuation ring of
$K\langle a \rangle$, with one of its coefficients equal to $1$. Then
$$v\big(\theta_\rho^mG_m(\mu_\rho+d_\rho)\big)=
m\gamma_\rho + v(G_{(\l)}(a)) + 
v\big(g_m(\boldsymbol{\sigma}(\mu_\rho+d_\rho))\big).$$
This calculation suggests a new constraint
on $\{\mu_\rho\}$, namely that for each $m\ge 1$ with 
$G_m\ne 0$,  
$$ v\big(g_m(\boldsymbol{\sigma}(\mu_\rho + d_\rho))\big)=0 
\qquad \text{(eventually in $\rho$)}. $$ 
Assume this constraint is met. Then Lemma~\ref{kapla} yields a fixed $m\ge 1$
such that 
if $m'\ge 1$ and $m'\ne m$, then, eventually in $\rho$,
$$v\big(\theta_\rho^mG_m(\mu_\rho+d_\rho)\big) <
v\big(\theta_\rho^{m'}G_{m'}(\mu_\rho+d_\rho)\big).$$
For this $m$ we have, eventually in $\rho$, 
$$v\big(G(b_\rho) - G(a)\big)= m\gamma_\rho +  v(G_{(\l)}(a)), 
\qquad \l=\l(m),$$ so
$G(b_\rho) \leadsto G(a)$, as desired.     

To have $\{\mu_\rho\}$ meet all constraints 
we introduce an axiom about $\ca{K}$ which involves only
the residue difference field $\bk$ of $\ca{K}$.
(More precisely, it is an {\em axiom scheme}.)

\vglue.3cm
\noindent
{\bf Axiom 3.}\ For each integer $d>0$ there is
$y\in \bk$ with $\bar{\sigma}^d(y)\ne y$. If $\text{char}(\bk)=p>0$, then
for any integers $d,e$ with $d\ne 0$ and $e> 0$ there is
$y\in \bk$ with $\bar{\sigma}^d(y)\ne y^{p^e}$. 

\medskip\noindent
By \cite{cohn}, p. 201, this axiom implies that there are
no residual $\sigma$-identities at all, that is, 
for every nonzero $f \in \bk[x_0,\ldots,x_n]$, there is
a $y$ in $\bk$ with $f(\bar{\boldsymbol{\sigma}}(y))\not=0$
(and thus the set $\{y\in \bk: f(\bar{\boldsymbol{\sigma}}(y))\not=0\}$ 
is infinite).
Even so, it may not be obvious that Axiom 3 allows us to 
select $\{\mu_\rho\}$ as
required, since the $g_m$'s are over $K\langle a\rangle,$ and we need
$\bar{\mu}_\rho\in \bk.$ Here is a well-known fact that will
take care of this:

\begin{lemma}\label{zariskitop} Let $k \subseteq k'$ be a field extension, and
$g(x_0,\dots,x_n)$ a nonzero polynomial over $k'$. Then there is
a nonzero polynomial $f(x_0,\dots,x_n)$ over $k$ such that whenever
$y_0,\dots,y_n\in k$ and
$f(y_0,\dots,y_n)\ne 0$, then 
$g(y_0,\dots,y_n) \ne 0$.
\end{lemma}
\begin{proof} Using a basis $b_1,\dots,b_m$ of the 
$k$-vector subspace of $k'$ generated by
the coefficients of $g$, we have $g = b_1 f_1 + \dots + b_mf_m$ with
$f_1,\dots,f_m\in k[x_0,\dots,x_n]$. Let $f$ be one of the nonzero $f_i$'s.
Then $f$ has the required property.
\end{proof}

\noindent
Consider an $m\ge 1$ with nonzero $G_m$, and define
$$ g_{m,\rho}(x_0,\dots,x_n):= 
g_m(x_0+d_\rho,\dots,x_n+\sigma^n(d_\rho)).$$
Then the reduced polynomial 
$$\bar{g}_{m,\rho}(x_0,\dots,x_n)=  
\bar{g}_m(x_0+\bar{d_\rho},\dots,x_n+\bar{\sigma}^n(\bar{d_\rho}))$$
is also nonzero for each $\rho$. 
By the lemma above we can pick a nonzero polynomial 
$f_{\rho}(x_0,\dots,x_n)\in \bk[x_0,\dots,x_n]$ such that if 
$y\in \mathcal{O}$ and
$f_{\rho}\big(\bar{\boldsymbol{\sigma}}(\bar{y})\big)\ne 0$, then 
$\bar{g}_{m,\rho}\big(\bar{\boldsymbol{\sigma}}(\bar{y}))\big) \ne 0$
for each  $m\ge 1$ with $G_m\ne 0$.

\medskip\noindent
{\bf Conclusion:} if for each $\rho$ 
the element $\mu_\rho\in \mathcal{O}$ satisfies $\bar{\mu}_\rho\ne 0$, 
$\bar{\mu}_\rho + \bar{d}_\rho \ne 0$, and
$f_{\rho}\big(\bar{\boldsymbol{\sigma}}(\bar{\mu_\rho})\big)\ne 0$, then
all constraints on $\{\mu_\rho\}$ are met.

\medskip\noindent
Axiom 3 allows us to meet these constraints, even if instead of a single
$G$ of order $\le n$ we have finitely many nonconstant 
$\sigma$-polynomials $G(x)$
of order $\le n$ and we have to meet simultaneously the constraints
coming from each of those $G$'s. This leads to:

\begin{theorem}\label{adjustment1}
Suppose $\ca{K}$ satisfies Axioms $2$ and $3$.  Suppose
$\{a_\rho\}$ in $K$ is a pc-sequence and $a_\rho \leadsto a$ in an
extension.  Let $\Sigma$ be a finite set of $\sigma$-polynomials $G(x)$
over $K$. Then there is a pc-sequence $\{b_\rho\}$ from $K$,
equivalent to $\{a_\rho\},$ such that $G(b_\rho)
\leadsto G(a)$ for each nonconstant $G$ in $\Sigma.$
\end{theorem}

\begin{corollary}
The same result, where $a$ is removed and one only asks that 
$\{G(b_\rho)\}$ is a pc-sequence. 
\end{corollary}
\begin{proof}
Put in an $a$ from an elementary extension.
\end{proof}

\bigskip\noindent
{\bf Refinement of the Basic Calculation.}
The following improvement of the basic calculation will be needed later on. 
Some minor differences with \cite{BMS} are because we shall use this in combination with a simpler notion of ``pc-sequence of $\sigma$-algebraic type''
and we do not assume $\ca{K}$ is unramified.

\begin{theorem} \label{crucial.result.nonwitt}
Suppose $\ca{K}$ satisfies Axioms $2$ and $3$. Let
$\{a_\rho\}$ be a pc-sequence from $K$ and let $a$ in some extension of $\ca{K}$ be such that $a_\rho \leadsto a$.
Let $G(x)$ be a $\sigma$-polynomial over $K$ such that
\begin{enumerate}
\item[(i)]
$G(a_\rho) \leadsto 0$, 
\item[(ii)] $ G_{(\l)}(b_\rho) \not\leadsto 0$ whenever $|\l|\geq 1$ and
$\{b_\rho\}$ is a pc-sequence in $K$ equivalent to $\{a_\rho\}$.
\end{enumerate}
Let $\Sigma$ be a finite set of $\sigma$-polynomials $H(x)$ over $K$. 
Then there is a
pc-sequence $\{b_\rho\}$ in $K$, equivalent to $\{a_\rho\}$, such that 
$G(b_\rho) \leadsto 0$, and
$H(b_\rho) \leadsto H(a)$ for every nonconstant $H$ in $\Sigma$.
\end{theorem}
\begin{proof} By augmenting $\Sigma$ we can assume that 
$G_{(\l)}\in \Sigma$ for all $\l$. Take $n$ such that all $H\in \Sigma$
have order $\le n$. 
Let $\{\theta_\rho\}$ and  $\{d_\rho\}$ be as before. 
Then the {\em basic calculation}
constructs nonzero polynomials 
$f_\rho\in \bk[x_0,\dots,x_n]$ such that if $\{\mu_\rho\}$ satisfies the
constraints 
$$\mu_\rho\in \mathcal{O},\ \bar{\mu}_\rho\ne 0,\ \bar{\mu}_\rho + \bar{d}_\rho \ne 0,\ 
 f_{\rho}\big(\bar{\boldsymbol{\sigma}}(\bar{\mu_\rho})\big)\ne 0,$$ 
then, setting $b_\rho:= a_\rho + \theta_\rho \mu_\rho$,
we have: $$H(b_\rho) \leadsto H(a)\text{ for each nonconstant }
H\in \Sigma .$$  
Let  $\{\mu_\rho\}$ satisfy these constraints, and define $\{b_\rho\}$
accordingly. Using Axiom 3 we shall be able to constrain $\{\mu_\rho\}$ 
further to achieve also $G(b_\rho) \leadsto 0$. We have
\begin{eqnarray*}
G(a_\rho)&=&G(b_\rho-\theta_\rho \mu_\rho)\\
&=&G(b_\rho)+\sum\limits_{m\geq 1}\sum\limits_{|\l|=m}
G_{(\l)}(b_\rho)\cdot\boldsymbol{\sigma}(-\theta_\rho \mu_\rho)^{\l}\\
&=&G(b_\rho)+\sum\limits_{m\geq 1}(-\theta_\rho)^m\cdot\sum\limits_{|\l|=m}
G_{(\l)}(b_\rho) \boldsymbol{\sigma}(\mu_\rho)^{\l}\\
&=&G(b_\rho)+\sum\limits_{m\geq 1}(-\theta_\rho)^m H_{m,\rho}(\mu_\rho)
\end{eqnarray*}
where $H_{m,\rho}$ is the $\sigma$-polynomial over $K$ defined by
$$  H_{m,\rho}(x)=
\sum\limits_{|\l|=m}
G_{(\l)}(b_\rho)\cdot \boldsymbol{\sigma}(x)^{\l}.$$
For $|\l|\ge 1$ we have $G_{(\l)}(b_\rho) \not \leadsto 0$
and, provided $G_{(\l)}\notin K$, 
$G_{(\l)}(b_\rho) \leadsto G_{(\l)}(a)$. Hence, 
for $|\l|\ge 1$ we have
$G_{(\l)}(b_\rho)=G_{(\l)}(a)(1+\epsilon_{\l,\rho})$, eventually in $\rho$,
where $v(\epsilon_{\l,\rho})>0$, and  
for our purpose we may assume that this holds  
for {\em all\/} $\rho$. Thus $H_{m,\rho}(x)= G_{m}(x) + \epsilon_{m,\rho}(x)$
where $G_m$ is as in {\em the basic calculation\/}: 
$$G_{m}(x) = \sum\limits_{|\l|=m}
G_{(\l)}(a)\cdot \boldsymbol{\sigma}(x)^{\l}, \quad 
\epsilon_{m,\rho}(x) = \sum\limits_{|\l|=m}
G_{(\l)}(a)\epsilon_{\l,\rho}\cdot \boldsymbol{\sigma}(x)^{\l}. $$ 
We put $\gamma_{\l}:=v\big(G_{(\l)}(a)\big)$, and restrict $m$ in what follows
to be $\ge 1$ with $G_m\ne 0$. For each $m$ 
we pick $\l=\l(m)$ with $|\l|=m$ for which
$\gamma_{\l}$ is minimal, so $G_m(x)=
G_{(\l)}(a)\cdot g_m\big(\boldsymbol{\sigma}(x)\big)$
where $g_m(x_0,\dots,x_n)$ has its coefficients 
in the valuation ring of
$K\langle a \rangle$, with one of its coefficients equal to $1$. As in 
{\em the basic calculation\/} we now impose the constraint on 
$\{\mu_\rho\}$ that for each $m$,
$$ v\big(g_m(\boldsymbol{\sigma}(\mu_\rho))\big)=0 
\qquad \text{(all $\rho$)}. $$ 
Then $v\big((-\theta_\rho)^mH_{m,\rho}(\mu_\rho)\big)=
m\gamma_\rho + \gamma_{\l(m)}$. Take the unique $m_0\ge 1$ with
 $G_{m_0}\ne 0$ such that if $m\ne m_0$, then
$$m_0\gamma_\rho + \gamma_{\l(m_0)} < m\gamma_\rho + \gamma_{\l(m)},
\quad \text{(eventually in $\rho$)}.$$
Again we can assume this holds for all $\rho$. Then for all $\rho$:
\begin{align*} 
v\big(&\sum\limits_{m}(-\theta_\rho)^m H_{m,\rho}(\mu_\rho)\big)= 
m_0\gamma_\rho +  \gamma_{\l(m_0)},\\
G(b_\rho)= G(a_\rho) - &\sum\limits_{m}(-\theta_\rho)^m H_{m,\rho}(\mu_\rho).
\end{align*} 
It follows that if $\rho$ is such that 
$v\big(G(a_\rho)\big) \ne  m_0\gamma_\rho + \gamma_{\l(m_0)}$, then 
$$ v\big(G(b_\rho)\big)= 
\min\left\{v\big(G(a_\rho)\big), m_0\gamma_\rho + \gamma_{\l(m_0)}\right\}.$$
We shall make this true for {\em all\/} $\rho$. 
Suppose that $v\big(G(a_\rho)\big) =  m_0\gamma_\rho + \gamma_{\l(m_0)}$. Then
$$G(b_\rho)=
G(a_\rho)\cdot\big(1- c_{\rho}g_{m_0}(\boldsymbol{\sigma}(\mu_\rho))+ 
\epsilon_\rho\big)$$
where $c_{\rho}=(-\theta_\rho)^{m_0} G_{(\l)}(a)/G(a_\rho)$, ($\l=\l(m_0)$), 
and $v( \epsilon_\rho) >0$. Note that $v(c_{\rho})=0$
and that $g_{m_0}$ is homogeneous of degree $m_0 >0$.
This leads to our final constraint on $\{\mu_\rho\}$: for each
$\rho$ such that $v\big(G(a_\rho)\big) =  
m_0\gamma_\rho + \gamma_{\l(m_0)}$ we impose
$$ 1-  \bar{c}_{\rho}\bar{g}_{m_0}
\big(\bar{\boldsymbol{\sigma}}(\bar{\mu}_\rho)\big)\ne 0.$$
Then $v\big(G(b_\rho)\big)= 
\min\left\{v\big(G(a_\rho)\big), m_0\gamma_\rho + \gamma_{l(m_0)}\right\}$
for all $\rho$, and thus $\left\{v\big(G(b_\rho)\big)\right\}$ 
is eventually strictly increasing. 
\end{proof}

\end{section}

\section{The Witt case}\label{witt}

\noindent
Let $\ca{K}=(K, \Gamma, \bk; v, \pi)$ be a valued difference field, 
satisfying Axiom 1 as usual. Assume that $\text{char}(K)=0$, 
$\text{char}(\bk)=p>0$, $\bk$ is perfect, that $\Gamma$ has a least positive element $1$ with $v(p)=1$, and, finally, that 
$\bar{\sigma}(y)=y^p$ on $\bk$.  
We call this the {\em Witt case\/} (for $p$).
These assumptions are satisfied by the Witt difference field $\operatorname{W}(\bk)$.

\medskip\noindent
Axiom 3 fails in the Witt case, but we shall adjust the basic calculation and its 
refinement to deal with this. As in \cite{BMS} we use the formalism of $\partial$-rings from \cite{joyal}.

\medskip\noindent
{\bf $\partial$-rings.}
Let $\partial_0:  \ca{O} \to \ca{O}$ be the identity map, and define
$$\partial_1: \ca{O} \to \ca{O}, \qquad \partial_1(x):=\ds{\frac{\sigma(x)-x^p}{p}}.$$ Usually $\partial_1$ is written as $\partial$; it satisfies 
the axioms for a $p$-derivation on
$\ca{O}$, namely
\begin{eqnarray*}
\partial(1)&=&0,\\
\partial(x+y)&=&\partial(x)+\partial(y)-\sum\limits_{i=1}^{p-1} 
a(p,i) x^i y^{p-i}, \quad a(p,i):={p\choose i}/p, \\
\partial(xy)&=&x^p\partial(y)+y^p\partial(x)+p\partial(x)\partial(y).
\end{eqnarray*}
A $\partial$-ring is a commutative ring with $1$ equipped with 
a unary
operation $\partial$ satisfying the above identities. 
For the basic facts on $\partial$-rings used below, see \cite{joyal}.
Because $\ca{O}$ is a $\partial$-ring, there
is a unique sequence of unary operations
$\partial_0,\partial_1,\partial_2,\ldots: \ca{O} \to \ca{O}$ with
$\partial_0,\partial_1$ as above such that for all $a\in \mathcal{O}$ and 
all $n$,
\begin{align*} \sigma^n(a)&= W_n(\partial_0(a),\dots, \partial_n(a)),\\
    W_n(x_0,\dots, x_n) &:= x_0^{p^n} + px_1^{p^{n-1}}+ \cdots + p^nx_n\in 
\mathbb{Z}[x_0,\dots, x_n].
\end{align*}
Recall that addition of Witt vectors \cite{serre} is given in terms of
polynomials 
\begin{align*} S_n\in \mathbb{Z}[y_0,\dots,y_n, z_0, \dots,z_n]\ &\text{ such that}\\
W_n(y_0,\dots, y_n)+ W_n(z_0,\dots, z_n)&=W_n(S_0,\dots, S_n),
\end{align*}
and accordingly, $\partial_n(a+b)= S_n\big(\partial_0(a),\dots, \partial_n(a),
\partial_0(b),\dots, \partial_n(b)\big)$ for all $a,b\in \ca{O}$.

In $\operatorname{W}[\bk]$, the ${\partial}_n$ yield the {\em components} of Witt vectors,
namely,
each $a\in \operatorname{W}[\bk]$ equals
$\big(\overline{\partial_0(a)},\overline{\partial_1(a)},\overline{\partial_2(a),}\ldots\big).$
In our Witt case, $\ca{O}/p^{n+1}\ca{O}\  \cong\ \operatorname{W}[\bk]/(p^{n+1})$:

\begin{lemma}\label{del.surjective} Identifying the vectors $(a_0,\dots,a_n)\in \bk^{n+1}$ with the elements of $\operatorname{W}[\bk]/(p^{n+1})$ in the usual way, we have a surjective ring morphism  
$$\ca{O} \to \operatorname{W}[\bk]/(p^{n+1}),\quad a\mapsto
\big(\overline{\partial_0(a)},
\overline{\partial_1(a)},\ldots,\overline{\partial_n(a)}\big) $$
with kernel $p^{n+1}\ca{O}$.
\end{lemma}

\noindent
A difference with \cite{BMS} is our use of the following in proving Theorem~\ref{adjustment2}.

\begin{lemma}\label{sum} Let $g\in \ca{O}[y_0,\dots,y_n]$ be such that 
its image $\bar{g}\in \bk[y_0,\dots,y_n]$ is nonzero. Then there is
a $g^*\in \ca{O}[y_0,\dots,y_n,z_0,\dots,z_n]$ such that
for all $a,b\in \ca{O}$, 
$$g\big(\partial_0(a+b),\dots, \partial_n(a+b)\big)= 
g^*\big(\partial_0(a),\dots, \partial_n(a),
\partial_0(b),\dots, \partial_n(b)\big),$$
and the image of 
$g^*\big(y_0,\dots,y_n,\partial_0(b),\dots, \partial_n(b)\big)$ in  
$\bk[y_0,\dots,y_n]$ is nonzero.
\end{lemma}
\begin{proof}
With the $S_n$ as above, put 
$g^*:= g(S_0,\dots, S_n)$. Then the displayed identity
holds. Let $b\in \ca{O}$ and put 
$h:= g^*\big(y_0,\dots,y_n,
\partial_0(b),\dots, \partial_n(b)\big)\in \ca{O}[y_0,\dots,y_n]$. 
In order to show that its image $\bar{h}$ in $\bk[y_0,\dots,y_n]$ 
is nonzero, we can assume that $\bk$ is
infinite (passing to a suitable Witt extension of $K$ if necessary).
Take $c_0,\dots,c_n\in \bk$ such that $\bar{g}(c_0,\dots,c_n)\ne 0$. 
By Lemma~\ref{del.surjective}, $(c_0,\dots,c_n)=(\overline{\partial_0(x)},
\overline{\partial_1(x)},\ldots,\overline{\partial_n(x)})$
for a suitable $x\in \ca{O}$. Let $a:= x-b$. Then by the above,
$$g\big(\partial_0(x),\dots,\partial_n(x)\big)=
h\big(\partial_0(a),\dots, \partial_n(a)\big),$$ with
image $\bar{g}(c_0,\dots,c_n)\ne 0$ in $\bk$. Thus $\bar{h}\ne 0$.
\end{proof}

\noindent
{\bf The $D$-transform.}
In analogy with $\boldsymbol{\sigma}$ and 
$\bar{\boldsymbol{\sigma}}$, we sometimes write 
$\boldsymbol{\partial}(a)$ for $(\partial_0(a),\ldots,\partial_n(a)),$ and 
$\bar{\boldsymbol{\partial}}(a)$ for
$\big(\overline{\partial_0(a)},\ldots,\overline{\partial_n(a)}\big)$ for 
$a$ in the valuation ring of a Witt extension.
Thus $\boldsymbol{\sigma}(a)=D(\boldsymbol{\partial}(a))$ for all such 
$a$, where
$$D(y_0,\ldots,y_n)=(y_0,y_0^p+py_1,\ldots,y_0^{p^n}+py_1^{p^{n-1}}+
\dots+p^ny_n).$$

\medskip\noindent
Let $F\in K[x_0,\dots,x_n]$ be homogeneous of degree $m>0$, and
consider its $D$-transform $F\big(D(y_0,\dots,y_n)\big)\in K[y_0,\dots,y_n]$. 
This $D$-transform is not in general homogeneous, but 
its constant term is zero and it has total
degree $\le mp^n$. Write
\begin{align*} F(x_0,\dots,x_n)&= \sum_{|\l|=m}a_{\l}\boldsymbol{x}^{\l},\quad 
(\text{all }a_{\l}\in K),\\ 
F\big(D(y_0,\dots,y_n)\big)&= 
\sum_{1\le |\j|\le mp^n} b_{\j}\boldsymbol{y}^{\j}, \quad
(\text{all }b_{\j}\in K).
\end{align*} 
To express how the $b_{\j}$ depend on the $a_{\l}$ we introduce a tuple 
$(x_{\l})$ of new variables, indexed by the $\l$ with $|\l|=m$.  

\begin{lemma}\label{lin} $b_{\j}=\Lambda_{\j,m}\big((a_{\l})\big)$ where
$\Lambda_{\j,m}\in \mathbb{Z}[(x_{\l})]$ is homogeneous 
of degree $1$ and depends only on
$\j,m$ and $p$, not on $K$ or $F$. 
\end{lemma}

\bigskip\noindent
{\bf Adjusting the Basic Calculation to the Witt Case.}
We now revisit the {\em basic calculation\/}
with the assumption that $\ca{K}=(K, \Gamma, \bk; v, \pi)$
is a Witt case satisfying Axiom 2,
with infinite $\bk$, and that
$a$ lies in a Witt case extension. As before $G(x)$ is a 
nonconstant $\sigma$-polynomial over $K$ of order $\le n$, and we have the
constraint on the sequence  $\{\mu_\rho\}$ in $\ca{O}$ that 
$$v(\mu_\rho)= v(\mu_\rho + d_\rho)=0 \quad \text{for all } \rho.$$ 
We now pick up the calculation at the point where 
$m\ge 1$, $G_m\ne 0$, and 
$$v\big(\theta_\rho^mG_m(\mu_\rho+d_\rho)\big)=
m\gamma_\rho + v(G_{(\l)}(a)) + 
v\big(g_m(\boldsymbol{\sigma}(\mu_\rho+d_\rho))\big).$$
Now 
$g_m\big(\boldsymbol{\sigma}(\mu_\rho+d_\rho)\big)= 
g_m\big(D(\boldsymbol{\partial}(\mu_\rho+d_\rho))\big)$, and the polynomial \\ 
$g_m\big(D(y_0,\dots,y_n)\big)$ over $K\langle a \rangle$ is nonzero, since
$D$ defines (in characteristic $0$) a generically surjective map from 
affine $(n+1)$-space to itself.
We have
$$g_{m}(D(y_0,\dots,y_n))=
\lambda_{m}\cdot g_{m}^D(y_0,\dots,y_n), 
\quad 0\ne \lambda_{m}\in K\langle a \rangle, $$
where the polynomial $g_{m}^D(y_0,\dots,y_n)$ 
is over the valuation ring of $K\langle a\rangle$ and has 
a coefficient equal to 1. Thus
$$ v\big(\theta_\rho^mG_m(\mu_\rho+d_\rho)\big)=
m\gamma_\rho + v(G_{(l)}(a)) + v(\lambda_m) + 
v\big(g_m^D(\boldsymbol{\partial}(\mu_\rho + d_\rho))\big).$$
This suggests that we constrain
$\{\mu_\rho\}$ such that for each $m\ge 1$ with 
$G_m\ne 0$,  
$$ v\big(g_m^D(\boldsymbol{\partial}(\mu_\rho + d_\rho))\big)=0 
\qquad \text{(eventually in $\rho$)}. $$ 
If this constraint is met, then $G(b_\rho) \leadsto G(a)$ 
as in the original calculation. We can meet the constraint as follows:
Lemma~\ref{sum} applied to $g_m^D$ yields for each $\rho$ a polynomial
$g_{m,\rho}(y_0,\dots,y_n)$
over the valuation ring of $K\langle a\rangle$ whose reduction 
$\bar{g}_{m,\rho}$ is nonzero such that 
$g_{m,\rho}(\boldsymbol{\partial}(c))= 
g_m^D\big(\boldsymbol{\partial}(c + d_\rho)\big)$ for all 
$c$ in the valuation ring
of $K\langle a\rangle$. Then by Lemma~\ref{zariskitop} we can pick 
for each $\rho$ a nonzero polynomial 
$f_\rho(y_0,\dots,y_n)\in \bk[y_0,\dots,y_n]$ such that if 
$c_0,\dots,c_n\in \ca{O}$ and
$f_{\rho}\big(\bar{c}_0,\dots,\bar{c}_n)\ne 0$, then 
$v\big(g_{m,\rho}(c_0,\dots,c_n)\big) = 0$ for all $m\ge 1$ such that
$G_m\ne 0$. 

\medskip\noindent
{\bf Conclusion:} if for each $\rho$ 
the element $\mu_\rho\in \ca{O}$ satisfies $\bar{\mu}_\rho\ne 0$,  
$\bar{\mu}_\rho + \bar{d}_\rho\ne 0$, and 
$f_{\rho}\big(\bar{\boldsymbol{\partial}}(\mu_\rho)\big)\ne 0$, then
all constraints on $\{\mu_\rho\}$ are met.

\medskip\noindent
Using Lemma~\ref{del.surjective}, it follows that all the constraints can 
be met. Thus:

\begin{theorem}\label{adjustment2}
Suppose $\ca{K}$ satisfies Axiom $2$ and is a Witt case
with infinite $\bk$.  Suppose $\{a_\rho\}$ is a pc-sequence from $K$, and
$a_\rho \leadsto a$ in a Witt case extension.  Then the conclusion
of Theorem \ref{adjustment1} holds, as does the Corollary to Theorem
\ref{adjustment1}.
\end{theorem}

\bigskip\noindent
{\bf Adjusting the Refinement of the Basic Calculation to the Witt Case.}
Let $\ca{K}$ 
be a Witt case with
$\bk$ of characteristic $p$. For a $\sigma$-polynomial $G(x)$ over $K$
of order $\le n$ and $a\in K$ we set
$$G(m,x):=\big(G_{(\l)}(x)\big)_{|\l|=m}, 
\quad G(m,a):=\big(G_{(\l)}(a)\big)_{|\l|=m}.$$ 
Note that if
$G$ is nonconstant, then
$\Lambda_{\j,m}\big(G(m,x)\big)$ has lower complexity than $G$ for  
$1 \le m\le \deg G$, $\j\in \mathbb{N}^{n+1}$, 
$1\le |\j| \le mp^n$, where $\Lambda_{\j,m}$ is as in Lemma~\ref{lin}.

\begin{theorem} \label{crucial.result.witt}
Suppose that $\ca{K}$ satisfies Axiom $2$, and is a 
Witt case with infinite $\bk$ of characteristic $p$. Let
$\{a_\rho\}$ be a pc-sequence from $K$ and $a_\rho \leadsto a$ 
with $a$ in a Witt case extension.
Let $G(x)$ be a $\sigma$-polynomial over $K$ of order $\le n$ so that
\begin{enumerate}
\item[(i)]
$G(a_\rho) \leadsto 0$; 
\item[(ii)]  $\Lambda_{\j,m}(G(m,b_\rho)) \not\leadsto 0$
whenever $1 \le m\le \deg G$, $\j\in \mathbb{N}^{n+1}$, 
$1\le |\j| \le mp^n$, and $\{b_\rho\}$ is a pc-sequence in $K$ 
equivalent to $\{a_\rho\}$. 
\end{enumerate}
Let $\Sigma$ be a finite set of $\sigma$-polynomials $H(x)$ over $K$. 
Then there is a
pc-sequence $\{b_\rho\}$ from $K$, equivalent to $\{a_\rho\}$, such that 
$G(b_\rho) \leadsto 0$, and
$H(b_\rho) \leadsto H(a)$ for each nonconstant $H\in \Sigma$.
\end{theorem}
\begin{proof} We can assume that 
$\Lambda_{\j,m}\big(G(m,x)\big)\in \Sigma$ whenever 
$1 \le m\le \deg G$ and $1\le |\j| \le mp^n$. By increasing $n$ if necessary 
we arrange that all $H(x)$ in $\Sigma$ have order $\le n$.
Let $\{\theta_\rho\}$ and  $\{d_\rho\}$ be as before. 
Then the {\em adjustment of the basic calculation to the Witt case\/}
constructs nonzero polynomials 
$f_\rho\in \bk[y_0,\dots,y_n]$ such that if $\{\mu_\rho\}$ satisfies the
constraints 
$$\mu_\rho\in \ca{O},\ \bar{\mu}_\rho\ne 0,\ \bar{\mu}_\rho + \bar{d}_\rho \ne 0,\ 
 f_{\rho}\big(\bar{\boldsymbol{\partial}}(\mu_\rho)\big)\ne 0,$$ 
then, setting $b_\rho:= a_\rho + \theta_\rho \mu_\rho$,
we have: 
$$H(b_\rho) \leadsto H(a), \text{ for each nonconstant }H\in \Sigma .$$  
Let  $\{\mu_\rho\}$ satisfy these constraints, and define $\{b_\rho\}$
accordingly. Proceeding as in the {\em refinement of the basic calculation}
we have
$$ G(a_\rho)\ =\ G(b_\rho)+\sum\limits_{m\geq 1}(-\theta_\rho)^m 
H_{m,\rho}(\mu_\rho)$$
where $H_{m,\rho}$ is the $\sigma$-polynomial over $K$ defined by
$$  H_{m,\rho}(x)=
\sum\limits_{|\l|=m}
G_{(\l)}(b_\rho)\cdot \boldsymbol{\sigma}(x)^{\l}.$$
Let $1 \le m\le \deg G$, and
 put $b_{\j,m,\rho}:= \Lambda_{\j,m}(G(m,b_\rho))$
for $1\le |\j|\le mp^n$. 
Then by Lemma~\ref{lin} and the facts stated just before it, 
\begin{align*} H_{m,\rho}(\mu_\rho)&= 
B_{m,\rho}\big(\boldsymbol{\partial}(\mu_\rho)\big), \text{ where}\\
B_{m,\rho}(\boldsymbol{y}):&=
\sum_{1\le |\j|\le mp^n} b_{\j,m,\rho}\boldsymbol{y}^{\j} 
\in K[y_0,\dots,y_n].
\end{align*}
For $1\le |\j|\le mp^n$ we put 
$b_{\j,m}:=  \Lambda_{\j,m}(G(m,a))$, 
and we may assume by (ii) that
$b_{\j,m,\rho}=b_{\j,m}(1+\epsilon_{\j,m,\rho})$ 
with $v(\epsilon_{\j,m,\rho})>0$, for all $\rho$. Thus 
\begin{align*}B_{m,\rho}(\boldsymbol{y})&= 
B_{m}(\boldsymbol{y}) + \epsilon_{m,\rho}(\boldsymbol{y}), \quad
\text{ where}\\ 
B_{m}(\boldsymbol{y}):&=\sum_{|\j|\le mp^n} b_{\j,m}\boldsymbol{y}^{\j}, \qquad 
\epsilon_{m,\rho}(\boldsymbol{y}) = \sum\limits_{|\j|\le mp^n}
b_{\j,m}\epsilon_{\j,m,\rho}\cdot \boldsymbol{y}^{\j}. 
\end{align*} 
In the rest of the argument, $m$ ranges over the natural numbers
$1,\dots,\deg G$ such that $B_m\ne 0$; put 
$\gamma_{\j,m}:= v(b_{\j,m})$ for $1\le |\j|\le mp^n$.
Pick $\j(m)\in \mathbb{N}^{n+1}$ with $1\le |\j(m)|\le mp^n$ such that 
$\gamma_{\j(m),m}= \min\{\gamma_{\j,m}:1\le |\j|\le mp^n\}$. Then 
$$H_{m,\rho}(\mu_\rho)= 
b_{\j(m),m}\cdot \left(h_{m}\big(\boldsymbol{\partial}(\mu_\rho)\big)+
\delta_{m,\rho}\right), $$
where $v(\delta_{m,\rho})>0$, and $h_{m}\in \ca{O}[y_0,\dots,y_n]$ 
has constant term zero and a coefficient equal to $1$. 
We now constrain
$\{\mu_\rho\}$ further by demanding, for each $m$ and $\rho$,
$$ \bar{h}_{m}\big(\bar{\boldsymbol{\partial}}(\mu_\rho)\big)\ne 0.$$
Then $v\big((-\theta_\rho)^m H_{m,\rho}(\mu_\rho)\big)=
m\gamma_\rho + \gamma_{j(m),m}$.
Take the unique $m_0\ge 1$ with $B_{m_0}\ne 0$ such that if $m\ne m_0$, then
$$m_0\gamma_\rho + \gamma_{\j(m_0),m_0} < m\gamma_\rho + \gamma_{\j(m),m},
\quad \text{(eventually in $\rho$)}.$$
Again we can assume this holds for all $\rho$. Then for all $\rho$:
\begin{align*} 
v\big(&\sum\limits_{m}(-\theta_\rho)^m H_{m,\rho}(\mu_\rho)\big)= 
m_0\gamma_\rho + \gamma_{j(m_0),m_0},\\
G(b_\rho)=G(a_\rho)-&\sum\limits_{m}(-\theta_\rho)^m H_{m,\rho}(\mu_\rho).
\end{align*}
Hence, if $\rho$ is such that 
$v\big(G(a_\rho)\big) \ne  m_0\gamma_\rho + \gamma_{\j(m_0),m_0}$, then 
$$ v\big(G(b_\rho)\big)= 
\min\left\{v\big(G(a_\rho)\big), 
m_0\gamma_\rho + \gamma_{\j(m_0),m_0}\right\}.$$
We shall make this true for {\em all\/} $\rho$. 
Suppose that $v\big(G(a_\rho)\big) =  m_0\gamma_\rho + \gamma_{\j(m_0),m_0}$. 
Then
$$G(b_\rho)=
G(a_\rho)\cdot\big(1- c_{\rho}h_{m_0}(\boldsymbol{\partial}(\mu_\rho))+ 
\epsilon_\rho\big)$$
where $c_{\rho}=(-\theta_\rho)^{m_0}\cdot b_{\j(m_0),m_0}/G(a_\rho)$, 
and $v( \epsilon_\rho) >0$. Note that $v(c_{\rho})=0$
and that $\bar{h}_{m_0}$ is not zero but that its constant term is zero.
This leads to our final constraint on $\{\mu_\rho\}$: for each
$\rho$ such that 
$v\big(G(a_\rho)\big) =  m_0\gamma_\rho + \gamma_{\j(m_0),m_0}$ 
we impose
$$ 1-  \bar{c}_{\rho}\bar{h}_{m_0}
\big(\bar{\boldsymbol{\partial}}(\mu_\rho)\big)\ne 0.$$
If all constraints hold, then $G(b_\rho) \leadsto 0$. 
Using Lemma~\ref{del.surjective} we can meet all constraints.
\end{proof}

\section{Newton-Hensel Approximation}\label{newtonhensel1}

\noindent
Let $\ca{K}=(K, \Gamma, \bk; v, \pi)$ be a valued difference field, satisfying Axiom 1 of course. Until Definition~\ref{shensel1} we fix a
$\sigma$-polynomial $G$ over $\ca{O}$ of order $\le n$, and let $a \in \ca{O}$.

\begin{definition}
$G$ is {\em $\sigma$-henselian at $a$\/} if
$v\big(G(a)\big) >0$ and   
$\min\limits_{|{\i}|=1}v\big(G_{(\i)}(a)\big)=0$.
\end{definition}

\noindent
The coefficients of all $G_{(\i)}$ are in $\ca{O}$. Hence, if
$G$ is $\sigma$-henselian at $a$, and $b\in \ca{O}$ satisfies $v(a-b)>0$, then
$G$ is also $\sigma$-henselian at $b$.  
If $G$ is $\sigma$-henselian at $a$ and $G(a) \neq 0$, does there 
exist $b\in \ca{O}$ such that $v(a-b)>0$ and $v(G(b))>v(G(a))$? To get a positive answer 
we use an additional assumption on $\bk$: 

\bigskip\noindent
{\bf Axiom $4_n$.} Each inhomogeneous linear $\bar{\sigma}$-equation 
$$1+\alpha_0x + \dots + \alpha_n\bar{\sigma}^n(x)=0 \qquad(\text{all }\alpha_i\in \bk, \text{ some }\alpha_i\ne 0),$$ 
has a solution in $\bk$. (And we say that $\ca{K}$ satisfies Axiom $4_n$ if $\bk$ does.)

\begin{lemma}\label{newton}
Suppose $\ca{K}$ satisfies Axiom $4_n$ and
$G$ is $\sigma$-henselian at $a$, with $G(a)\ne 0$.
Then there is $b \in \ca{O}$ such that $v(a-b)\ge v\big(G(a)\big)$ 
and $v\big(G(b)\big)>v\big(G(a)\big)$. For any such $b$ we have $v(a-b)=v\big(G(a)\big)$ and $G$ is
$\sigma$-henselian at $b$.
\end{lemma}

\begin{proof} 
Let $b=a + G(a)u$ where $u \in \ca{O}$ is to be determined later. Then 
$$G(b)=G(a)+\sum\limits_{|\i|\geq
1}G_{(\i)}(a)\cdot\boldsymbol{\sigma}(G(a)u)^{\i}.$$
Extracting a factor $G(a)$ and using Axiom 1 it follows that
$$G(b)= G(a)\cdot\big( 1+\sum_{|\i|=1} c_{\i}\boldsymbol{\sigma}(u)^{\i} + 
\sum_{|\j|>1}c_{\j}\boldsymbol{\sigma}(u)^{\j}\big)$$
where $\min\limits_{\i|=1} v(c_{\i})=0$ and 
$v(c_{\j})>0$ for $|\j|>1$.  
Using Axiom $4_n$, we can pick $u \in \ca{O}$ such that $\bar{u}$ is a 
solution of 
$$1+ \sum\limits_{|\i|=1}\overline{c_{\i}}\cdot \bar{\boldsymbol{\sigma}}(x)^{\i}=0.$$
Then $v(b - a)=v(G(a))$, and
$v(G(b)) > v(G(a))$. It is clear that any $b \in \ca{O}$ with $v(a-b)\ge v\big(G(a)\big)$ 
and $v\big(G(b)\big)>v\big(G(a)\big)$ is obtained in this way.
\end{proof}

\begin{lemma}\label{hensel.imm} Suppose $\ca{K}$
satisfies Axiom $4_n$ and $G(x)$ is $\sigma$-henselian at $a$.
Suppose also that there is no $b\in K$ with $G(b)=0$ and $v(a-b) = v(G(a))$.
Then there is a pc-sequence $\{a_\rho\}$ in $K$ 
with the following properties:
\begin{enumerate}
\item $a_0=a$ and $\{a_\rho\}$ has no pseudolimit in $K$;
\item $\{v(G(a_\rho))\}$ is strictly increasing, and thus 
 $G(a_\rho) \leadsto 0$;
\item $v(a_{\rho'} -a_\rho)=v\big(G(a_\rho)\big)$ whenever $\rho<\rho'$;
\item for any extension $\ca{K}'=(K',\dots)$ of $\ca{K}$ and $b,c\in K'$ such that $a_\rho \leadsto b$, $G(c)=0$ and $v(b-c)\ge v(G(b))$, we have $a_\rho \leadsto c$.
\end{enumerate}
\end{lemma}  
\begin{proof} Let $\{a_\rho\}_{\rho<\lambda}$ be a sequence 
in $\ca{O}$ with $\lambda$ an ordinal $>0$, $a_0=a$, and
\begin{enumerate}
  \item[(i)] $G$ is $\sigma$-henselian at $a_\rho$, for all $\rho < \lambda$,
  \item[(ii)] $v(a_{\rho'} -a_\rho)=v\big(G(a_\rho)\big)$ 
whenever $\rho<\rho'<\lambda$,
  \item[(iii)] $v(G(a_{\rho')})>v(G(a_\rho))$ whenever $\rho<\rho'<\lambda$.
\end{enumerate}
(Note that for $\lambda=1$ we have such a sequence.)
Suppose $\lambda= \mu +1$ is a successor ordinal. Then Lemma~\ref{newton} 
yields $a_\lambda\in K$ such that $v(a_\lambda - a_\mu)=v\big(G(a_\mu)\big)$
and $v\big(G(a_\lambda)\big)> v\big(G(a_\mu)\big)$. Then the extended
sequence $\{a_\rho\}_{\rho<\lambda +1}$ has the above properties 
with $\lambda+1$ instead of $\lambda$. 

Suppose $\lambda$ is a limit ordinal. Then $\{a_\rho\}$ is
a pc-sequence and 
$G(a_\rho) \leadsto 0$. If $\{a_\rho\}$ has no pseudolimit in $K$ 
we are done. Assume otherwise, and take
a pseudolimit $a_\lambda\in K$ of $\{a_\rho\}$. The extended 
sequence $\{a_\rho\}_{\rho<\lambda +1}$ clearly satisfies the conditions 
(i) and (ii) with $\lambda+1$ instead of $\lambda$. Since
$G$ is over $\ca{O}$ we have 
$$v(G(a_\lambda) - G(a_{\rho} )) \ge v(a_\lambda - a_{\rho})=
v(a_{\rho +1}-a_\rho)= v (G(a_{\rho}))$$ for 
$\rho < \lambda$. 
Therefore $v(G(a_\lambda)) \ge v(G(a_\rho))$ for $\rho<\lambda$, and by (iii)
this yields $v(G(a_\lambda))>v(G(a_\rho))$ for 
$\rho<\lambda$. So the extended sequence also satisfies (iii) with $\lambda+1$ instead of $\lambda$. For cardinality reasons this building process must come to an end and thus yield a pc-sequence $\{a_\rho\}$ satisfying (1), (2), (3). 

Let $b,c$ in an extension of $\ca{K}$ be such that $a_\rho \leadsto b$, $G(c)=0$ and $v(b-c)\ge v(G(b))$.  
Then $G$ is $\sigma$-henselian at $b$, and for $\rho < \rho'$,
\begin{align*} \gamma_\rho:&= v(a_{\rho'}-a_\rho)=v(G(a_\rho))=v(b-a_\rho), \text{ so}\\
   v(b-c)& \ge v(G(b))=v\big(G(b)-G(a_\rho) + G(a_\rho)\big)\ge \gamma_\rho,
   \end{align*}
since $v(G(b)-G(a_\rho))\ge v(b-a_\rho)=\gamma_\rho$. Thus $a_\rho \leadsto c$, as claimed.
\end{proof}

\begin{definition}\label{shensel1} We say
$\ca{K}$ is {\em $\sigma$-henselian\/} if for each $\sigma$-polynomial
$G(x)$ over $\ca{O}$ and $a\in \ca{O}$ such that $G$ is 
$\sigma$-henselian at $a$, 
there exists $b \in \ca{O}$ such that $G(b)=0$ and $v(a-b)\ge v\big(G(a)\big)$. {\em(By the arguments above, any such $b$ will actually satisfy 
$v(a-b)= v\big(G(a)\big)$.)}
\end{definition}

\begin{corollary}\label{fixh} If $\ca{K}$ is $\sigma$-henselian, then
the residue field of $\operatorname{Fix}(K)$ is $\operatorname{Fix}(\bk)$.
\end{corollary}
\begin{proof} Suppose $\ca{K}$ is $\sigma$-henselian, and let $\alpha\in \operatorname{Fix}(\bk)$; we shall find 
$b\in \operatorname{Fix}(K)$ such that $v(b)=0$ and $\bar{b}=\alpha$. 
Take $a\in K$ with $v(a)=0$  and $\bar{a}=\alpha$. Then $v(\sigma(a)-a)>0$, so
$\sigma(x)-x$ is $\sigma$-henselian at $a$. So there is a $b$ as promised.
\end{proof}

\noindent
By {\bf Axiom $4$} we mean the axiom scheme
$\{\operatorname{Axiom} 4_n:\  n=0,1,2,\dots\}$. 

\begin{remark}\label{scan} If
$\Gamma = \{0\}$, then $\ca{K}$ is $\sigma$-henselian. Suppose 
$\Gamma \ne \{0\}$, $\ca{K}$ satisfies 
Axiom $2$ and is $\sigma$-henselian. Then $\ca{K}$
satisfies Axiom $4$ by \cite{scanlon1}, Proposition 5.3, so
$\bar{\sigma}^n\ne \operatorname{id}_{\bk}$ for all $n\ge 1$. Hence, if also
$\cha(\bk)=0$, then $\ca{K}$ satisfies Axiom $3$. 
\end{remark}

\noindent
From part (1) of Lemma~\ref{hensel.imm} we obtain:

\begin{corollary}\label{disc} If $\ca{K}$ is maximal as valued field and satisfies Axiom $4$, then $\ca{K}$ is $\sigma$-henselian. In particular, if
$\ca{K}$ is complete with discrete valuation and satisfies Axiom $4$, then
$\ca{K}$ is $\sigma$-henselian.
\end{corollary}

\noindent
Thus if the difference field $\bk$ satisfies Axiom 4, then the Hahn difference field
$\bk((t^\Gamma))$ is $\sigma$-henselian.  
Suppose $\bk$ has characteristic $p>0$ and every equation
$$1+ \alpha_0x + \alpha_{1}x^{p}+ \cdots + \alpha_nx^{p^n}=0 \qquad  (\text{all }\alpha_i\in \bk, \text{ some }\alpha_i\ne 0),$$
is solvable in $\bk$. Then by Corollary~\ref{disc} the Witt
difference field $\operatorname{W}(\bk)$ is $\sigma$-henselian, where $\sigma$ is the
Witt frobenius. As noted in \cite{BMS}, this
condition on the residue field $\bk$ is {\em
Hypothesis A} in Kaplansky~\cite{kaplansky} where it is 
related to uniqueness of maximal immediate
extensions of valued fields. It is equivalent to $\bk$
not having
any field extension of finite degree divisible by $p$; see
\cite{whaples}.

\medskip\noindent
Note that if $\ca{K}$ is
$\sigma$-henselian, then it is henselian as a valued field.
We have the following analogue of an
important result about henselian valued fields:

\begin{theorem}\label{lift.res.field} Suppose that 
$\ca{K}$ is 
$\sigma$-henselian and $\operatorname{char}(\bk)=0$.
Let $K_0\subseteq \ca{O}$ be a $\sigma$-subfield of $K$. Then there is
a $\sigma$-subfield $K_1$ of $K$ such that
$K_0\subseteq K_1 \subseteq \ca{O}$ and $\bar{K}_1=\bk$.
\end{theorem}
\begin{proof}
Suppose that $\bar{K_0}\ne \bk$. Take $a\in \ca{O}$ such that 
$\bar{a}\notin \bar{K_0}$. If $v(G(a))=0$ for all nonzero
$G(x)$ over $K_0$, then $K_0\langle a\rangle$ is a proper
$\sigma$-field extension of $K_0$ contained in $\ca{O}$. Next,
consider the case that $v\big(G(a)\big)>0$  for some nonzero
$G(x)$ over $K_0$. Pick such $G$ of minimal complexity.  So $v(H(a))=0$
for all nonzero $H(x)$ over $K_0$ of lower complexity.  It follows that
$G$ is $\sigma$-henselian at $a$. 
So there is $b\in \ca{O}$ with $G(b)=0$ and 
$v(a-b)=v(G(a))$, so $\bar{a}=\bar{b}$. We claim that $K_0\langle b \rangle$ 
is a proper $\sigma$-field extension of $K_0$ contained in $\ca{O}$.
To prove the claim, let $G$ have order $m$. Then the $\sigma^k(b)$ with
$k\in \mathbb{Z}$ are algebraic over $K_0\big(b,\dots, \sigma^{m-1}(b)\big)$ and thus 
$$K_0\langle b\rangle\ =\ K_0\big(\sigma^k(b): k\in \mathbb{Z}\big)\ =\ K_0\big(b,\dots, \sigma^{m-1}(b)\big)[\sigma^k(b): k\in \mathbb{Z}] \subseteq \ca{O},$$
which establishes the claim. We finish the proof by Zorn's Lemma. 
\end{proof}

\noindent
The notion ``$\sigma$-henselian at $a$'' 
applies only to $\sigma$-polynomials over $\ca{O}$ and $a\in \ca{O}$.
It will be convenient to extend it a little. Let $G(x)$ be over $K$
of order $\le n$ and $a\in K$.

\begin{definition}
We say $(G,a)$ is in {\em $\sigma$-hensel
configuration\/} if $G_{(\i)}(a)\ne 0$ for some 
$\i \in \mathbb{N}^{n+1}$ with $|\i|=1$, and either
$G(a)=0$ or there is $\gamma\in \Gamma$ such that
$$ v\big(G(a)\big)= \min_{|\i|=1}v\big(G_{(\i)}(a)\big) + \gamma 
<v\big(G_{(\j)}(a)\big) + |\j|\cdot\gamma$$
for all $\j$ with $|\j|>1$. For $(G,a)$ in $\sigma$-hensel configuration we put 
$$\gamma(G,a):= v\big(G(a)\big) - \min_{|\i|=1}v\big(G_{(\i)}(a)\big).$$
\end{definition}

\noindent
Let $(G,a)$ be in $\sigma$-hensel configuration, $G(a)\ne 0$, and
take $c\in K$ with $v(c)=\gamma(G,a)$ and
put $H(x):= G(cx)/G(a)$, $\alpha:= a/c$. Then 
$$H(\alpha)=1, \quad \min_{|\i|=1} v\big(H_{(\i)}(\alpha)\big)=0, \quad v\big(H_{(\j)}(\alpha)\big) >0  \text{ for }|\j|>1,$$ 
as is easily verified. In
particular, $H(\alpha +x)$ is over $\ca{O}$.
Now assume that $\ca{K}$ satisfies also Axiom 4.
This gives a unit $u\in \ca{O}$ such that
$v\big(H(\alpha+u)\big)>0$. We claim that $H(\alpha +x)$ is 
$\sigma$-henselian at $u$. This is because for $P(x):=H(\alpha +x)$ we 
have $v\big(P(u)\big)>0$, and 
for each $i$,
$$P_{(\i)}(u)= H_{(\i)}(\alpha+u)= H_{(\i)}(\alpha)+ 
\sum_{|\j|\ge 1}H_{(\i)(\j)}(\alpha)\boldsymbol{\sigma}(u)^{\j},$$
so $\min_{|\i|=1} v\big(P_{(\i)}(u)\big)=0$.
If  $\ca{K}$ is $\sigma$-henselian, we can take
$u$ as above such that $H(\alpha+u)=0$, and then $b:= a+cu$ satisfies
$G(b)=0$ and $v(a-b)=\gamma$. Summarizing:

\begin{lemma}\label{hensel-conf} 
Assume $\ca{K}$ satisfies Axiom $4$ and is
$\sigma$-henselian. Let $(G,a)$ be in $\sigma$-hensel configuration.
Then there is $b\in K$ such that $G(b)=0$ and 
$v(a-b)= \gamma(G,a)$.
\end{lemma} 

\noindent
Up to this point
we treated the Witt and non-Witt case separately, but
from now on it makes sense to handle both cases at once. We say that 
$\ca{K}$ is {\em workable} if either it satisfies Axioms 2, 3  (as in
Theorem~\ref{crucial.result.nonwitt}), or it satisfies Axiom 2
and is a Witt case with infinite $\bk$ (as in
Theorem~\ref{crucial.result.witt}). An {\em extension\/} 
of a workable $\ca{K}$ is an extension as before 
(and doesn't have to be workable), but
in the Witt case we also require the extension to be a Witt case.

In the next definition we assume that $\ca{K}$ is workable, 
and that $\{a_\rho\}$ is a pc-sequence from $K$.

\begin{definition}
We say $\{a_\rho\}$ is of
{\em $\sigma$-algebraic type over $K$\/} if $G(b_\rho) \leadsto 0$
for some $\sigma$-polynomial $G(x)$ over $K$ and an equivalent 
pc-sequence $\{b_\rho\}$ in $K$.

If $\{a_\rho\}$ is of $\sigma$-algebraic type over $K$, then a 
{\em minimal $\sigma$-polynomial of  $\{a_\rho\}$ over $K$\/} is a 
$\sigma$-polynomial $G(x)$ over $K$ with the following properties: 
\begin{enumerate}
\item[(i)] $ G(b_\rho) \leadsto 0$ for some pc-sequence 
$\{b_\rho\}$ in $K$, 
equivalent to $\{a_\rho\}$;
\item[(ii)] $ H(b_\rho) \not\leadsto 0$ whenever $H(x)$ 
is a $\sigma$-polynomial 
over $K$ of lower complexity than $G$ and 
$\{b_\rho\}$ is a pc-sequence in $K$ equivalent to $\{a_\rho\}$.
\end{enumerate}
\end{definition}

\noindent
If  $\{a_\rho\}$ is of $\sigma$-algebraic type over $K$, then 
$\{a_\rho\}$ clearly has a minimal $\sigma$-polynomial over $K$.
The next lemma is used to study immediate extensions 
in the next section. Its finite ramification hypothesis is satisfied by all
Witt cases. (``Finitely ramified'' is defined right after Lemma~\ref{kapla}.)

\begin{lemma}\label{henselconf}
Suppose $\ca{K}$ is workable and finitely ramified. 
Let $\{a_\rho\}$ from $K$ be a pc-sequence of
$\sigma$-algebraic type over $K$ with minimal $\sigma$-polynomial  
$G(x)$ over $K$, and with pseudolimit $a$ in some extension.
Let $\Sigma$ be a finite set of $\sigma$-polynomials $H(x)$ 
over $K$. Then there is a pc-sequence
$\{b_\rho\}$ in $K$, equivalent to $\{a_\rho\}$, such that, 
with $\gamma_\rho:=v(a-a_\rho)$:
\begin{enumerate}
\item[$(1)$] $v(a-b_\rho)=\gamma_\rho$, eventually, and 
$ G(b_\rho) \leadsto 0$;
\item[$(2)$] if $H\in \Sigma$ and $H \notin K$, then  
$H(b_\rho) \leadsto H(a)$;
\item[$(3)$] $(G,a)$ is in  $\sigma$-hensel configuration, and 
$\gamma(G,a) > \gamma_\rho$, eventually.
\end{enumerate} 
\end{lemma}
\begin{proof}
Let $G$ have order $n$. We can assume that $\Sigma$ includes
all $G_{(\i)}$.
In the rest of the proof $\i,\j,\l$ range over $\mathbb{N}^{n+1}$.
Theorems~\ref{crucial.result.nonwitt} and ~\ref{crucial.result.witt} and
their proofs 
yield an equivalent pc-sequence 
$\{b_\rho\}$ in $K$ such that (1) and (2) hold. The proof of
Theorem~\ref{adjustment1} 
shows that we can arrange that in addition 
there is a unique $m_0$ with $1\le m_0 \leq \deg G$ such that, eventually,
$$v\big(G(b_\rho)-G(a)\big)=
\min\limits_{|\i|=m_0}v\big(G_{(\i)}(a)\big)
+m_0\gamma_\rho < v\big(G_{(\j)}(a)\big)+|\j|\cdot\gamma_\rho,$$
for each $\j$ with $|\j|\ge 1$ and $|\j|\not=m_0$. 
Now $\left\{v\big(G(b_\rho)\big)\right\}$ is strictly increasing, 
eventually, so $v\big(G(a)\big) > v\big(G(b_\rho)\big)$ eventually, 
and for $|\j|\ge 1$, $|\j|\not=m_0$: 
$$v\big(G(b_\rho)\big)=\min\limits_{|\i|=m_0}v\big(G_{(\i)}(a)\big)
+m_0\cdot \gamma_\rho < v\big(G_{(\j)}(a)\big) + |\j|\cdot\gamma_\rho,
\quad \text{eventually}.$$  
We claim that $m_0=1$. Let $|\i|=1$ with $G_{(\i)}\ne 0$, and let $\j>\i$;
our claim will then follow by deriving 
$$v\big(G_{(\i)}(a)\big) + \gamma_\rho < v\big(G_{(\j)}(a)\big) + 
|\j|\gamma_\rho, \quad \text{eventually}.$$ The proof of
Theorem~\ref{adjustment1} with $G_{(\i)}$ in the role of $G$ shows that 
we can arrange that our sequence $\{b_\rho\}$ also satisfies
$$v\big(G_{(\i)}(b_\rho)-G_{(\i)}(a)\big) \le  v\big(G_{(\i)(\l)}(a)\big)
+|\l|\cdot\gamma_\rho, \quad\text{eventually}$$
for all $\l$ with $|\l|\ge 1$. Since 
$v\big(G_{(\i)}(b_\rho)\big)=v\big(G_{(\i)}(a)\big)$ eventually, this yields
$$v\big(G_{(\i)}(b_\rho)\big)\leq v\big(G_{(\i)(\l)}(a)\big)
+|\l|\cdot\gamma_\rho, \quad\text{eventually}$$
for all $\l$ with $|\l|\ge 1$, hence for all such $\l$,
$$v\big(G_{(\i)}(b_\rho)\big)\leq
 v{{\i+ \l}\choose \i}
+v\big(G_{(\i+ \l)}(a)\big)+|\l|\cdot\gamma_\rho, \quad\text{eventually} $$
For $\l$ with $\i+\l=\j$, this yields
$$v\big(G_{(\i)}(a))\leq v{{\j\choose \i}}
+v\big(G_{(\j)}(a)\big)+(|\j|-1)\cdot\gamma_\rho, \quad\text{eventually}.$$
Now $K$ is finitely ramified, so 
$$v\big(G_{(\i)}(a))<v\big(G_{(\j)}(a)\big)+(|\j|-1)\cdot
\gamma_\rho, \quad\text{eventually, hence}  $$
$$v\big(G_{(\i)}(a)\big)+\gamma_\rho<v(G_{(\j)}(a))+
|\j|\cdot\gamma_\rho, \quad\text{eventually}. $$
Thus $m_0=1,$ as claimed. The above inequalities 
then yield (3). 
\end{proof}

\section{Immediate Extensions}\label{max.imm.ext}

\noindent
Throughout this section
$\ca{K}=(K,\Gamma,\bk; v,\pi)$ is a workable valued difference field. The immediate 
extensions of $\ca{K}$ are then workable as well, and we prove here the basic facts on these immediate extensions. To avoid 
heavy handed notation we often let $K$ stand for $\ca{K}$ 
when the context permits. 

\begin{definition}
A pc-sequence $\{a_\rho\}$ from $K$ is said to be of 
{\em $\sigma$-transcendental type over $K$\/}
if it is not of $\sigma$-algebraic type over $K$, that is, 
$G(b_\rho) \not\leadsto 0$ for each $\sigma$-polynomial $G(x)$
over $K$ and each equivalent pc-sequence $\{b_\rho\}$ from $K$. 
\end{definition}

\noindent
In particular, such a pc-sequence cannot have a pseudolimit in $K$.
The next lemma is a $\sigma$-analogue
of a result familiar for valued fields.

\begin{lemma}\label{ext.with.pseudolimit} 
Let $\{a_\rho\}$ from $K$ be a pc-sequence of $\sigma$-transcendental type
over $K$.  Then  $\ca{K}$ has an
immediate extension $(K\langle a \rangle, \Gamma,\bk; v_a, \pi_a)$
such that: \begin{enumerate} 
\item[$(1)$] $a$ is $\sigma$-transcendental over $K$ and 
$a_\rho \leadsto a$;
\item[$(2)$] for any extension $(K_1,\Gamma_1,\bk_1;v_1,\pi_1)$ of 
$\ca{K}$ and any $b\in K_1$ with
$a_\rho \leadsto b$ there is a unique
embedding 
$$(K\langle a \rangle, \Gamma,\bk; v_a, \pi_a)\ \longrightarrow\ 
(K_1,\Gamma_1,\bk_1;v_1,\pi_1)$$ 
over  $\ca{K}$ that sends $a$ to $b$. 
\end{enumerate} 
\end{lemma}
\begin{proof}
Let $\ca{K}'$ be an elementary extension
of $\ca{K}$ containing a pseudolimit $a$ of $\{a_\rho\}$.  
Let  $(K\langle a \rangle, \Gamma_a,\bk_a; v_a, \pi_a)$
be the valued $\sigma$-field generated by $a$ over $\ca{K}$. 
To prove that $\Gamma_a=\Gamma$, consider a
nonconstant $\sigma$-polynomial $G(x)$ over $K$.
Use Theorem \ref{adjustment1} to get an equivalent
$\{b_\rho\}$ so that $G(b_\rho) \leadsto  G(a).$
Now, $G(b_\rho) \not\leadsto 0$, since $\{a_\rho\}$
is of $\sigma$-transcendental type.  So 
$v_a(G(a))=$ eventual value of $v(G(b_\rho))\in\Gamma.$ 
Thus $\Gamma_a=\Gamma$ and $a$ is $\sigma$-transcendental over $(K, \sigma)$. 
A similar argument shows that $\bk_a=\bk$. 

With $b$ as in (2), the proof of Theorem \ref{adjustment1}
shows that in the argument above we can arrange, in addition to 
$G(b_\rho) \leadsto  G(a)$, that
$G(b_\rho) \leadsto  G(b)$; hence $v_a(G(a))= v_1(G(b))\in \Gamma$.  
\end{proof}

\noindent
The following consequence involves both (1) and (2):

\begin{corollary}\label{ctra} Let $a$ from some extension of
$\ca{K}$ be $\sigma$-algebraic over $K$ and let 
$\{a_\rho\}$ be a  pc-sequence in $K$ such that $a_\rho \leadsto a$.
Then $\{a_\rho\}$ is of $\sigma$-algebraic type over $K$.
\end{corollary}

\noindent
The $\sigma$-algebraic analogue of Lemma~\ref{ext.with.pseudolimit} is
trickier:

\begin{lemma}\label{imm.alg.ext} 
Suppose $\ca{K}$ is finitely ramified. Let $\{a_\rho\}$ from $K$ be a pc-sequence of $\sigma$-algebraic type over
$K$, with no pseudolimit in $K$. Let $G(x)$ be a minimal 
$\sigma$-polynomial of $\{a_\rho\}$ over $K$. Then $\ca{K}$ 
has an immediate extension 
$(K\langle a \rangle, \Gamma,\bk; v_a, \pi_a)$ such that
\begin{enumerate}
\item[$(1)$] $G(a)=0$ and $a_\rho \leadsto a$;
\item[$(2)$] for any extension $(K_1, \Gamma_1,\bk_1; v_1, \pi_1)$ of 
$\ca{K}$ and any $b\in K_1$ with $G(b)=0$ and
$a_\rho \leadsto b$ there is a unique
embedding 
$$(K\langle a \rangle, \Gamma,\bk; v_a, \pi_a) \longrightarrow\ 
(K_1, \Gamma_1,\bk_1; v_1, \pi_1)$$ 
over  $\ca{K}$ that sends $a$ to $b$. 
\end{enumerate} 
\end{lemma}

\begin{proof}
Let $G(x)=F(\boldsymbol{\sigma}(x))$ with  
$F(x_0, \dots, x_n) \in K[x_0, \dots, x_n]$, $n=\text{order}(G)$.

{\em Claim}. $F$ is irreducible in $K[x_0,\dots,x_n]$. Suppose 
otherwise. Then $F=F_1F_2$ with nonconstant $F_1,F_2\in K[x_0,\dots,x_n]$, 
and thus $G=G_1G_2$ where $G_1(x),G_2(x)$ are $\sigma$-polynomials over
$K$ of lower complexity than $G$. Take a pc-sequence $\{b_\rho\}$ in $K$ 
equivalent to 
$\{a_\rho\}$ such that $G(b_\rho) \leadsto 0$ and $\{G_1(b_\rho)\}$, 
$\{G_2(b_\rho)\}$ 
pseudoconverge. Then $\{v(G(b_\rho))\}$ is eventually strictly increasing, 
but $\{v(G_1(b_\rho))\}$, $\{v(G_2(b_\rho))\}$ are eventually constant, 
contradiction.

Consider the domain $K[\xi_0,\dots,\xi_n]:=K[x_0,\dots,x_n]/(F)$
with $\xi_i:= x_i + (F)$ and let $L=K(\xi_0,\dots,\xi_n)$ be its
field of fractions.
We extend the valuation $v$ on $K$ to a valuation $v: L^\times \to \Gamma$ as
follows.  Pick a pseudolimit 
$e$ of $\{a_\rho\}$ in some extension of $\ca{K}$ 
whose valuation we also denote by $v$. Let
$\phi\in L$, $\phi\ne 0$, so 
$\phi=f(\xi_0,\dots,\xi_n)/g(\xi_0,\dots,\xi_{n-1})$
with $f\in K[x_0,\ldots,x_n]$ of lower $x_n$-degree
than $F$ and  $g\in K[x_0,\ldots,x_{n-1}]$, $g\ne 0$. 
We claim:
$v\big(f(\boldsymbol{\sigma}(e))\big), 
v\big(g(\boldsymbol{\sigma}(e))\big)\in \Gamma$, and 
$v\big(f(\boldsymbol{\sigma}(e))\big)- v\big(g(\boldsymbol{\sigma}(e))\big)$ 
depends only on
$\phi$ and not on the choice of $(f,g)$. To see why this claim is true, 
suppose that
also $\phi=f_1(\xi_0,\dots,\xi_n)/g_1(\xi_0,\dots,\xi_{n-1})$
with $f_1\in K[x_0,\ldots,x_n]$ of lower $x_n$-degree
than $F$ and  $g_1\in K[x_0,\ldots,x_{n-1}]$, $g_1\ne 0$. Then 
$fg_1\equiv f_1g \mod F$ in $K[x_0,\dots,x_n]$, and thus $fg_1=f_1g$ since
$fg_1$ and $f_1g$ have lower degree in $x_n$ than $F$. To avoid some tedious 
case distinctions we assume that $f,g, f_1, g_1$ are all nonconstant. 
(In the other cases the arguments below need some trivial modifications.)   
Take a pc-sequence 
$\{b_\rho\}$ in $K$ equivalent to 
$\{a_\rho\}$ such that  
$\{f(\boldsymbol{\sigma}(b_\rho))\}\leadsto f(\boldsymbol{\sigma}(e))$, 
and likewise with $g, f_1, g_1$ instead of $f$.
Also $\{f(\boldsymbol{\sigma}(b_\rho))\}\not\leadsto 0$ by the minimality 
of $G$, so 
$$v(f(\boldsymbol{\sigma}(e)))= \text{eventual value of }
v(f(\boldsymbol{\sigma}(b_\rho))),$$
in particular, $v(f(\boldsymbol{\sigma}(e)))\in \Gamma$,
and likewise with $g$, $f_1$ and $g_1$ instead of $f$. The 
identity $fg_1=f_1g$ now yields 
$$v\big(f(\boldsymbol{\sigma}(e))\big)-v\big(g(\boldsymbol{\sigma}(e))\big)=
v\big(f_1(\boldsymbol{\sigma}(e))\big)-
v\big(g_1(\boldsymbol{\sigma}(e))\big).$$
This proves the claim and allows us to define $v: L^\times \to \Gamma$
by 
$$v(\phi):=v\big(f(\boldsymbol{\sigma}(e))\big)-
v\big(g(\boldsymbol{\sigma}(e))\big).$$
It is routine to check that this map $v$ is a valuation on the field $L$
that extends the valuation $v$ on $K$. (For the multiplicative law
$v(\phi_1\phi_2)=v(\phi_1) + v(\phi_2)$, let
$$\phi_1=\frac{f_1(\xi_0,\dots,\xi_n)}{g_1(\xi_0,\dots,\xi_{n-1})}, \quad \phi_2=\frac{f_2(\xi_0,\dots,\xi_n)}{g_2(\xi_0,\dots,\xi_{n-1})}, \quad
 \phi_1\phi_2=\frac{f_3(\xi_0,\dots,\xi_n)}{g_3(\xi_0,\dots,\xi_{n-1})}
$$
where $f_1, f_2, f_3\in K[x_0,\ldots,x_n]$ have lower $x_n$-degree
than $F$ and where $g_1,g_2, g_3$ are nonzero polynomials
in $K[x_0,\ldots,x_{n-1}]$.
In view of
$$\frac{f_1}{g_1}\frac{f_2}{g_2}=\frac{QF}{g_1g_2g_3} + \frac{f_3}{g_3}$$ 
with $Q \in K[x_0, \dots ,x_n]$, we obtain 
$v(\phi_1\phi_2)=v(\phi_1) + v(\phi_2)$ as in the proof on  pp. 308--309  
of \cite{kaplansky} for ordinary valued fields, using suitable choices of 
pc-sequences equivalent to $\{a_\rho\}$ in the style above.) 
Likewise one shows that $(L,v)$ has the same residue field as $(K,v)$.

It is clear that $K(\xi_0,\dots,\xi_{n-1})$
is purely transcendental over $K$ of
transcendence degree $n$. The same is true for 
$K(\xi_1,\dots,\xi_n)$: by the minimality of $G$ the variable
$x_0$ must occur in $F$, so $\xi_0$ is algebraic over 
$K(\xi_1,\dots,\xi_n)$.  
This yields an isomorphism 
$$ K(\xi_0,\dots,\xi_{n-1})\buildrel{\sigma}\over\longrightarrow
K(\xi_1,\dots,\xi_n), 
\quad \sigma(\xi_i)=\xi_{i+1}\text{ for }0\leq i\leq n-1, $$
between subfields of $L$ that extends
$\sigma$ on $K$. We consider these subfields as equipped with the valuation
induced by that of $L$, and claim that $\sigma$ is an isomorphism of
{\em valued\/} fields.
To see why, let $c\in K[\xi_0,\dots,\xi_{n-1}]$, $c\ne 0$; the claim will 
follow by deriving $v(c)=v(\sigma(c))$. We have
$$c=h(\xi_0,\dots,\xi_{n-1}), \quad  h\in K[x_0,\dots,x_{n-1}].$$
Let $h^\sigma\in K[x_0,\dots,x_{n-1}]$ be obtained by applying 
$\sigma$ to the coefficients of $h$. Then 
$\sigma(c)=h^\sigma(\xi_1,\dots,\xi_n)$. Also
$\sigma(c)=f(\xi_0,\dots,\xi_n)/g(\xi_0,\dots,\xi_{n-1})$
with $f\in K[x_0,\ldots,x_n]$ of lower $x_n$-degree
than $F$ and  $g\in K[x_0,\ldots,x_{n-1}]$, $g\ne 0$.
Thus 
$$ g(x_0,\dots,x_{n-1})h^\sigma(x_1,\dots,x_n)-f(x_0,\dots,x_n)=qF$$
with $q\in K[x_0,\dots,x_n]$. Put $\alpha:= v(f(\xi_0,\dots,\xi_n))$,
$\beta=v(g(\xi_0,\dots,\xi_{n-1}))$ and 
$\gamma=v(h(\xi_0,\dots,\xi_{n-1}))=v(c)$, so 
$\alpha,\beta, \gamma\in \Gamma$ and $v(\sigma(c))=\alpha - \beta$.
Take a pc-sequence $\{b_\rho\}$ from $K$ 
equivalent to $\{a_\rho\}$
such that, eventually 
$$v\big(f(\boldsymbol{\sigma}(b_\rho))\big)=\alpha, \quad
v\big(g(\boldsymbol{\sigma}(b_\rho))\big)=\beta, \quad
v\big(h(\boldsymbol{\sigma}(b_\rho))\big)=\gamma$$
and also $G(b_\rho) \leadsto 0$ and
$\{v\big(q(\boldsymbol{\sigma}(b_\rho))\big)\}$ is eventually constant.
Now, to be explicit, 
$h(\boldsymbol{\sigma}(b_\rho))=h(b_\rho, \dots, \sigma^{n-1}(b_\rho))$,
so $\sigma\big (h(\boldsymbol{\sigma}(b_\rho))\big)= 
h^\sigma(\sigma(b_\rho),\dots,\sigma^n(b_\rho))$, and thus
$v\big(h^\sigma(\sigma(b_\rho),\dots,\sigma^n(b_\rho))\big)=\gamma$, 
eventually. Since 
$\{v\big((qF)(\boldsymbol{\sigma}(b_\rho))\big)\}$ is 
either eventually strictly increasing, or eventually equal to $\infty$,
it follows that eventually
$$\beta+ \gamma= v\left(g\big(\boldsymbol{\sigma}(b_\rho)\big)\cdot 
h^\sigma\big(\sigma(b_\rho),\dots,\sigma^n(b_\rho)\big)\right) = 
v\left(f\big(\boldsymbol{\sigma}(b_\rho)\big)\right)=\alpha,$$
so $\alpha - \beta=\gamma$, and thus $v(\sigma(c))=v(c)$. This
proves our claim.

Consider the inclusion diagram of valued fields
(with $L^h$ the henselisation of $L$):

$$\begin{array}{ccccc}
 & & L^h & & \\
 & & \uparrow & & \\
 & & L & & \\
 & \nearrow & & \nwarrow & \\
K(\xi_0,\ldots,\xi_{n-1}) & & & & K(\xi_1,\ldots,\xi_n)\\
 & \nwarrow &  & \nearrow & \\
 & & K & &
\end{array}
$$
Note that $L$ is an algebraic immediate extension of both 
$K(\xi_0,\dots,\xi_{n-1})$ and 
$K(\xi_1,\dots,\xi_n)$, so the same is true for $L^h$ instead of $L$.
Since $K$ is finitely ramified---a hypothesis we 
use here for the first time in the proof---this gives
$K(\xi_0,\dots,\xi_{n-1})^h=K(\xi_1,\dots,\xi_n)^h=L^h$ where we take
the henselizations inside $L^h$.
So $\sigma$ extends uniquely to an automorphism $\sigma$ of the {\em valued\/}
field $L^h$. Put $a:= \xi_0$ and let 
$(K\langle a \rangle,\Gamma, \bk; v_a, \pi_a)$ 
be the valued $\sigma$-subfield of $L^h$ generated by $a$ over $K$.  
Note that then $G(a)=0$ and $a_\rho \leadsto a$, since
$v_a(a_\rho-a)=v(a_\rho - c)$. 

To verify (2), the first display in the proof shows that  
the valuation on $L$ defined above does not depend on the choice of the 
pseudolimit $e$.
So we can take for $e$ an element $b$ as in the hypothesis of (2), from which
the conclusion of (2) follows. 
\end{proof}

\noindent
We note the following consequence.
\begin{corollary} \label{immediate.sigma-immediate}
Suppose $\ca{K}$ is finitely ramified. Then $\ca{K}$ as a valued field has 
a proper immediate extension if and only if
$\ca{K}$ as a valued difference field has a proper immediate extension.
\end{corollary}

\noindent
We say that $\ca{K}$ is 
{\em $\sigma$-algebraically maximal\/} if it
has no proper immediate $\sigma$-algebraic extension, and we say it is
{\em maximal\/} if it has no proper immediate extension. 
Corollary~\ref{ctra} and Lemmas~\ref{imm.alg.ext} and ~\ref{hensel.imm} yield:

\begin{corollary} Suppose $\ca{K}$ is finitely ramified. Then: 
\begin{enumerate}
\item[$(1)$] $\ca{K}$ is $\sigma$-algebraically maximal 
if and only if each pc-sequence in $K$ of $\sigma$-algebraic type over $K$ 
has a pseudolimit in $K$;
\item[$(2)$] if $\ca{K}$ satisfies Axiom $4$ and is
$\sigma$-algebraically maximal, then $\ca{K}$ is 
$\sigma$-henselian.
\end{enumerate}
\end{corollary}

\noindent
It is clear that $\ca{K}$ has
$\sigma$-algebraically maximal immediate $\sigma$-algebraic extensions, and 
also maximal immediate extensions. If $\ca{K}$ satisfies Axiom 4 both kinds of extensions are
unique up to isomorphism, but for this we need one more lemma: 

\begin{lemma} Suppose $\ca{K}$ is finitely ramified and
$\ca{K}'$ is a workable finitely ramified
$\sigma$-algebraically maximal
extension of $\ca{K}$ satisfying Axiom $4$. Let
$\{a_\rho\}$ from $K$ be a pc-sequence of $\sigma$-algebraic type over
$K$, with no pseudolimit in $K$, and with minimal 
$\sigma$-polynomial $G(x)$ over $K$. Then there exists $b\in K'$
such that $a_\rho \leadsto b$ and $G(b)=0$.
\end{lemma}
\begin{proof} Lemma~\ref{imm.alg.ext} provides a pseudolimit $a\in K'$ of
$\{a_\rho\}$. Take a pc-sequence $\{b_\rho\}$ in $K$ 
equivalent to $\{a_\rho\}$ with the properties listed in 
Lemma~\ref{henselconf}. 
Since $\ca{K}'$ is $\sigma$-henselian 
and satisfies Axiom 4, Lemma~\ref{hensel-conf} yields $b\in K'$ such that
$$ v'(a-b)=\gamma(G,a)\ \text{  and }\ G(b)=0.$$ 
Note that $a_\rho \leadsto b$ since
$\gamma(G,a) > v(a-a_\rho)=\gamma_\rho$ eventually.
\end{proof} 

\noindent
Together with Lemmas~\ref{ext.with.pseudolimit} and ~\ref{imm.alg.ext}
this yields:

\begin{theorem}\label{unique.max.imm.ext} Suppose $\ca{K}$ 
is finitely ramified and
satisfies Axiom $4$. Then all its maximal immediate extensions
are isomorphic over
$\ca{K}$, and all its $\sigma$-algebraically maximal immediate 
$\sigma$-algebraic extensions are isomorphic over
$\ca{K}$.
\end{theorem}

\noindent
We now state minor variants of these results using
the notion of saturation from model theory, as needed 
in the proof of the embedding theorem in the next
section. Let $|X|$ denote the cardinality of a set $X$, and let $\kappa$ be a 
cardinal.

\begin{lemma} Suppose $\ca{E}=(E, \Gamma_E, \dots)\le \ca{K}$ is workable and $\ca{K}$ is finitely ramified, $\sigma$-henselian, and $\kappa$-saturated with
$\kappa>|\Gamma_E|$. 
Let $\{a_\rho\}$ from $E$ be a pc-sequence of $\sigma$-algebraic type over
$E$, with no pseudolimit in $E$, and with minimal
$\sigma$-polynomial $G(x)$ over $E$. Then there exists $b\in K$
such that $a_\rho \leadsto b$ and $G(b)=0$.
\end{lemma}
\begin{proof} By the saturation assumption we have a pseudolimit $a\in K$ of
$\{a_\rho\}$. Let $\gamma_\rho=v(a-a_\rho)$. By Lemma~\ref{henselconf}, $(G,a)$ is in $\sigma$-hensel
configuration with $\gamma(G,a)>\gamma_\rho$, eventually.
Since  $\ca{K}$ is $\sigma$-henselian, it satisfies
Axiom 4, so Lemma~\ref{hensel-conf} yields $b\in K$ such that
$v(a-b)=\gamma(G,a)$ and $G(b)=0$. Note that $a_\rho \leadsto b$ since
$\gamma(G,a) > \gamma_\rho$ eventually.
\end{proof}

\noindent
In combination with Lemmas~\ref{ext.with.pseudolimit} and 
~\ref{imm.alg.ext} this yields:

\begin{corollary}\label{immsat} If $\ca{E}=(E, \Gamma_E, \dots)\le \ca{K}$ is 
workable and satisfies Axiom $4$, and $\ca{K}$ is finitely ramified, 
$\sigma$-henselian, and  $\kappa$-saturated with
$\kappa>|\Gamma_E|$, then any maximal immediate 
extension of $\ca{E}$ can be embedded in $\ca{K}$ over $\ca{E}$.
\end{corollary}

\begin{section}{The Equivalence Theorem}\label{et}

\noindent
Theorem~\ref{embed11}, the main result of the paper,
tells us when two workable $\sigma$-henselian valued difference fields
of equal characteristic zero are elementarily 
equivalent over a common substructure.  
In Section 8 we derive from it in the 
usual way some attractive consequences on the elementary theories of
such valued difference fields and on the induced structure on value group and
residue difference field. In Section 9 we use coarsening to obtain 
analogues in the mixed characteristic case. 

We begin with a short subsection on angular component maps. The presence of 
such maps simplifies the proof of the Equivalence Theorem, but in the aftermath 
we can often discard these maps again, by Corollary~\ref{acm2}.

\medskip\noindent
{\bf Angular components.} Let $\ca{K}=(K, \Gamma, \bk;\ v,\pi)$ be a valued
difference field. An {\em angular component map\/} on $\ca{K}$ is
an angular component map $\ac$ on $\ca{K}$ as valued field such that in 
addition $\bar{\sigma}(\ac(a))=\ac(\sigma(a))$ for all $a\in K$. Examples are
the Hahn difference fields $\bk((t^\Gamma))$ with angular component map given 
by $\ac(a)=a_{\gamma_0}$ for
nonzero $a=\sum a_\gamma t^\gamma\in \bk((t^\Gamma))$ and $\gamma_0=v(a)$, 
and also the Witt difference fields
$\operatorname{W}(\bk)$ with angular component map determined by $\ac(p)=1$.
(To see this, use the next lemma and the fact that 
$\text{Fix}\big(\operatorname{W}(\bk)\big)=\operatorname{W}(\mathbb{F}_p)=\mathbb{Q}_p$.)

\begin{lemma}\label{acm1} Suppose $\ca{K}$ satisfies Axiom $2$. Then
each angular component map on the valued subfield $\operatorname{Fix}(K)$ of 
$\ca{K}$ 
extends uniquely to an angular component map on $\ca{K}$. If in addition
$\ca{K}$ is $\sigma$-henselian, then every angular
component map on $\ca{K}$ is obtained in this way from an
angular component map on $\operatorname{Fix}(K)$.
\end{lemma}
\begin{proof} Given an angular component map $\ac$ on $\operatorname{Fix}(K)$
the claimed extension to $\ca{K}$, also denoted by $\ac$, is obtained as 
follows: for $x\in K^\times$
we have $x=uy$ with $u,y\in K^\times,\ v(u)=0, \sigma(y)=y$; then
$\ac(x)=\bar{u}\ac(y)$. The second claim of the lemma follows from
Corollary~\ref{fixh}. 
\end{proof}

\noindent
Here is an immediate consequence of Lemmas~\ref{xs1} and \ref{acm1}: 

\begin{corollary}\label{acm2} 
Suppose $\ca{K}$ satisfies Axiom $2$. Then there is an 
angular component map on some elementary extension of $\ca{K}$.
\end{corollary}

\medskip\noindent
{\bf The Main Result.}
In this subsection we consider $3$-sorted structures
$$ \ca{K}=\big(K, \Gamma, \bk; v, \pi, \ac \big)$$
where $\big(K, \Gamma, \bk; v, \pi \big)$ is a valued difference
field (satisfying Axiom 1 of course) and 
where $\ac: K \to \bk$ is an angular component map on 
$\big(K, \Gamma, \bk; v, \pi \big)$.
Such a structure will be called an
{\em ac-valued difference field\/}. Any subfield $E$ of $K$ is
viewed as a valued subfield of $\ca{K}$ with valuation ring
$\ca{O}_E:=\ca{O}\cap E$.

\medskip\noindent
If $\cha(\bk)=0$ and $\ca{K}$ is
$\sigma$-henselian, then by Theorem~\ref{lift.res.field} there is a difference
ring morphism $i:\bk \to \ca{O}$ such that $\pi(i(a))=a$ for all $a \in K$;
we call such $i$ a {\em $\sigma$-lifting\/} of $\bk$ to  $\ca{K}$. This will play a minor
role in the proof of the Equivalence Theorem.

\medskip\noindent
A {\em good substructure\/} of 
$\ca{K}=(K, \Gamma, \bk; v,\pi,\ac)$ is a triple
$\ca{E}=(E, \Gamma_{\ca{E}}, \bk_{\ca{E}})$ such that \begin{enumerate}
\item $E$ is a difference subfield of $K$, 
\item $\Gamma_{\ca{E}}$ is an ordered abelian subgroup of $\Gamma$ 
with $v(E^\times)\subseteq \Gamma_{\ca{E}}$,
\item $\bk_{\ca{E}}$ is a difference subfield of $\bk$ with 
$\ac(E)\subseteq \bk_{\ca{E}}$ (hence $\pi(\ca{O}_E)\subseteq \bk_{\ca{E}}$).
\end{enumerate}
 For good substructures $\ca{E}_1=(E_1, \Gamma_1, \bk_1)$ and 
$\ca{E}_2=(E_2, \Gamma_2, \bk_2)$ of $\ca{K}$, we define 
$\ca{E}_1\subseteq \ca{E}_2$ to mean that 
$E_1 \subseteq E_2,\ \Gamma_1 \subseteq \Gamma_2,\ \bk_1 \subseteq \bk_2$.
If $E$ is a difference subfield of $K$ with 
$\ac(E)=\pi(\ca{O}_E)$, then 
$\big(E, v(E^\times), \pi(\ca{O}_E)\big)$ is a good substructure of $\ca{K}$, 
and if in addition $F\supseteq E$ is a difference subfield
of $K$ such that $v(F^\times)=v(E^\times)$, then $\ac(F)=\pi(\ca{O}_F)$.
Throughout this subsection 
$$\ca{K}=(K, \Gamma, \bk; v,\pi, \ac), \qquad \ca{K}'=(K',
\Gamma', \bk'; v', \pi', \ac')$$ are ac-valued difference fields, with 
valuation rings $\ca{O}$ and $\ca{O}'$, and
$$\ca{E}=(E,\Gamma_{\ca{E}}, \bk_{\ca{E}}), \qquad \ca{E'}=(E',\Gamma_{\ca{E}'},
\bk_{\ca{E}'})$$ are good substructures of
$\ca{K}$, $\ca{K'}$ respectively. To avoid too many accents we let $\sigma$ 
denote the difference operator of each of $K, K', E, E'$, and put 
$\ca{O}_{E'}:= \ca{O}'\cap E'$.

 A {\em good map\/} $\mathbf{f}: \ca{E} \to \ca{E'}$
is a triple $\mathbf{f}=(f, f_{\val}, f_{\res})$ consisting
of a difference field isomorphism
$f:E \to E'$, an ordered group isomorphism $f_{\val}:\Gamma_{\ca{E}} \to
\Gamma_{\ca{E}'}$ and a difference field isomorphism $f_{\res}: \bk_{\ca{E}} \to
\bk_{\ca{E}'}$ such that
\begin{enumerate}
\item[(i)] $f_{\val}(v(a))=v'(f(a))$ for all $a \in E^\times$, and
$f_{\val}$ is elementary as a partial map between the
ordered abelian groups $\Gamma$ and $\Gamma'$; 
\item[(ii)]
$f_{\res}(\ac(a))=\ac'(f(a))$ for all $a \in E$, and
$f_{\res}$ is elementary as a partial map between the
difference fields $\bk$ and $\bk'$.
\end{enumerate}
Let $\mathbf{f}: \ca{E} \to \ca{E'}$ be a good map as above. Then the 
field part $f: E \to E'$ of $\af$ is a valued difference field isomorphism, and
$f_{\val}$ and $f_{\res}$ agree on $v(E^\times)$ and $\pi(\ca{O}_E)$ with
the maps $v(E^\times) \to v'(E'^\times)$ and 
$\pi(\ca{O}_E) \to \pi'(\ca{O}_{E'})$ induced by $f$. 
We say that a good map $\ag= (g, g_{\val}, g_{\res}) : \ca{F} \to \ca{F'}$ 
{\em extends\/} $\af$ if $\ca{E}\subseteq \ca{F}$, $\ca{E'}\subseteq \ca{F'}$, 
and $g$, $g_{\val}$, $g_{\res}$ extend $f$, $f_{\val}$, $f_{\res}$, respectively.  
The {\em domain\/} of $\mathbf{f}$ is $\ca{E}$.

The next two lemmas show that condition (ii) above is automatically
satisfied by certain extensions of good maps.

\begin{lemma}\label{ac1} Let $\af: \ca{E} \to \ca{E}'$ be a good map, and
$F\supseteq E$ and $F'\supseteq E'$ subfields of $K$ and
$K'$, respectively, such that $v(F^\times)=v(E^\times)$ and 
$\pi(\ca{O}_F)\subseteq \bk_{\ca{E}}$. Let $g: F \to F'$ be a valued field 
isomorphism such that $g$ extends $f$ and $f_{\res}(\pi(u))=\pi'(g(u))$
for all $u\in \ca{O}_F$. Then $\ac(F)\subseteq  \bk_{\ca{E}}$
and $f_{\res}(\ac(a))= \ac'(g(a))$ for all $a \in F$.  
\end{lemma}
\begin{proof} Let $a \in F$. Then $a=a_1u$ where $a_1 \in E$ and  
$u\in \ca{O}_F$, $v(u)=0$, so $\ac(a)=\ac(a_1)\pi(u)\in \bk_{\ca{E}}$. 
It follows easily that $f_{\res}(\ac(a))= \ac'(g(a))$.
\end{proof}

\noindent
In the same way we obtain: 

\begin{lemma}\label{ac2} Suppose $\pi(\ca{O}_E)=\bk_{\ca{E}}$, let 
$\af: \ca{E} \to \ca{E}'$ be a good map, and let
$F\supseteq E$ and $F'\supseteq E'$ be subfields of $K$ and
$K'$, respectively, such that $v(F^\times)=v(E^\times)$. 
Let $g: F \to F'$ be a valued field isomorphism extending 
$f$.   
Then $\ac(F)= \pi(\ca{O}_F)$ and
$g_{\res}(\ac(a))= \ac'(g(a))$ for all $a \in F$, where the map
$g_{\res}:\pi(\ca{O}_F)\to \pi'(\ca{O}_{F'})$ is induced by $g$
$($and thus extends $f_{\res})$. 
\end{lemma}

\medskip\noindent
The following is useful in connection with Axiom 2:

\begin{lemma}\label{fixr} Let $b\in K^{\times}$. Then the following are
equivalent:
\begin{enumerate}
 \item[$(1)$] There is $c \in \Fix(K)$ such that $v(c)=v(b)$. 
 \item[$(2)$] There is $d \in K$
such that $v(d)=0$ and $\sigma(d)=(b/ \sigma(b))\cdot d$.
\end{enumerate}
\end{lemma}
\begin{proof} 
For $c$ as in (1), $d=cb^{-1}$ is as in (2). For $d$ as in (2), $c=bd$ is as in (1).
\end{proof}

\noindent
We say that $\ca{E}$ satisfies Axiom 2 (respectively, Axiom 3, Axiom 4) 
if the valued difference subfield $(E, v(E^\times),\pi(\ca{O}_E);\dots)$ of 
$\ca{K}$ does. Likewise, we say that $\ca{E}$ is workable (respectively, 
$\sigma$-henselian) if this valued difference subfield of $\ca{K}$ is.

\begin{theorem}\label{embed11} Suppose $\cha(\bk)=0$, $\ca{K}$, $\ca{K}'$ 
satisfy Axiom $2$ and are $\sigma$-henselian. Then any good map
$\ca{E} \to \ca{E}'$ is a partial elementary map between $\ca{K}$ and 
$\ca{K}'$.
\end{theorem}    
\begin{proof} The theorem holds trivially for $\Gamma=\{0\}$, so assume 
that $\Gamma\ne \{0\}$. Then $\ca{K}$ and $\ca{K}'$ are workable.
Let $\mathbf{f}=(f, f_{\val}, f_{\res}): \ca{E} \to \ca{E}'$
be a good map. By passing to suitable elementary extensions of 
$\ca{K}$ and $\ca{K}'$ we arrange that
$\ca{K}$ and $\ca{K}'$ are $\kappa$-saturated, where 
$\kappa$ is an uncountable cardinal such that 
$|\bk_{\ca{E}}|,\ |\Gamma_{\ca{E}}| < \kappa$.  
Call a good substructure $\ca{E}_1=(E_1,\bk_1, \Gamma_1)$ of $\ca{K}$ 
{\em small\/} if $|\bk_1|,\ |\Gamma_1|<\kappa$.
We shall prove that the good maps with small domain form a back-and-forth
system between $\ca{K}$ and $\ca{K}'$. (This clearly suffices to obtain 
the theorem.) In other words, we shall prove that under the 
present assumptions on $\ca{E}$, $\ca{E}'$ and $\mathbf{f}$, there is   
for each $a \in K$ a good map $\ag$ extending $\af$ such that
$\ag$ has small domain $\ca{F}=(F,\dots)$ with $a\in F$.

In addition to Corollary~\ref{immsat}, we have several basic extension 
procedures:

\medskip\noindent
(1) {\em Given $\alpha\in \bk$, arranging that $\alpha\in \bk_{\ca{E}}$}.
By saturation and the definition of ``good map'' this can be achieved without 
changing $f$, $f_{\val}$, $E$, $\Gamma_{\ca{E}}$
by extending $f_{\res}$ to a partial elementary map between $\bk$ and $\bk'$ 
with $\alpha$ in its domain.

\medskip\noindent
(2) {\em Given $\gamma\in \Gamma$, arranging that $\gamma\in \Gamma_{\ca{E}}$}.
This follows in the same way.  

\medskip\noindent
(3) {\em Arranging $\bk_{\ca{E}}=\pi(\ca{O}_E)$}. Suppose 
$\alpha\in \bk_{\ca{E}},\ \alpha\notin \pi(\ca{O}_E)$; set 
$\alpha':= f_{\res}(\alpha)$. 

If $\alpha$ is $\bar{\sigma}$-transcendental over $\pi(\ca{O}_E)$,
we pick $a\in \ca{O}$ and $a'\in \ca{O}'$ such
that $\bar{a}=\alpha$ and $\bar{a'}=\alpha'$, and then
Lemmas~\ref{extrest} and ~\ref{ac1} yield a 
good map $\ag=(g, f_{\val}, f_{\res})$ with small domain
$(E \langle a \rangle, \Gamma_{\ca{E}}, \bk_{\ca{E}})$ 
such that $\ag$ extends $\af$ and $g(a)=a'$.

Next, assume that $\alpha$ is $\bar{\sigma}$-algebraic
over $\pi(\ca{O}_E)$. Let $G(x)$ be a $\sigma$-polynomial over $\ca{O}_E$
such that $\bar{G}(x)$ is a minimal
$\bar{\sigma}$-polynomial of $\alpha$ over $\pi(\ca{O}_E)$ and 
has the same complexity
as $G(x)$. Pick $a\in \ca{O}$ such
that $\bar{a}=\alpha$. Then $G$ is $\sigma$-henselian at $a$. So we have $b
\in \ca{O}$ such that $G(b)=0$ and $\bar{b}=\bar{a}=\alpha$. 
Likewise, we obtain $b'\in \ca{O}'$ such that $f(G)(b')=0$ and 
$\bar{b'}=\alpha'$, where $f(G)$ is
the difference polynomial over $E'$ that corresponds to $G$ under $f$.
By Lemmas~\ref{extresa} and \ref{ac1} we obtain a good map extending $\af$ 
with small domain $(E \langle b \rangle,  \Gamma_{\ca{E}}, \bk_{\ca{E}})$ and 
sending $b$ to $b'$. 

\bigskip\noindent
By iterating these steps we can arrange 
$\bk_{\ca{E}}=\pi(\ca{O}_E)$; this condition 
is actually preserved in the extension procedures (4), (5), (6) below, 
as the reader may easily verify. We do assume in the rest of the proof that  
$\bk_{\ca{E}}=\pi(\ca{O}_E)$, and so we can refer from now on to
$\bk_{\ca{E}}$ as the {\em residue difference field\/} of $E$.

\medskip\noindent
(4) {\em Extending $\af$ to a good map whose domain satisfies Axiom $2$}.
Let $\delta \in v(E^\times)$. Pick $b \in E^{\times}$ such
that $v(b)=\delta$. Since Axiom 2 holds in $\ca{K}$, we can use 
Lemma~\ref{fixr}
to get $d \in K$ such that $v(d)=0$ and $G(d)=0$ where
$$G(x)\ :=\  \sigma(x) - \frac{b}{\sigma(b)}\cdot x.$$ Note that
$v(qd)=0$ and $G(qd)=0$ for all $q \in \mathbb{Q}^{\times} \subseteq
E^{\times}$. Hence by saturation we can assume that $v(d)=0$, $G(d)=0$
and $\bar{d}$ is transcendental
over $\bk_{\ca{E}}$. We set $\alpha= \bar{d}$, so $\bar{G}(x)$ is a minimal
$\bar{\sigma}$-polynomial of $\alpha$ over $\bk_{\ca{E}}$. 
By Lemma~\ref{extresa},
$$E\langle d \rangle=E(d), \qquad v(E(d)^\times)=v(E^\times), 
\qquad \pi(\ca{O}_{E(d)})=\bk_{\ca{E}}(\alpha), \qquad
\sigma (E(d))=E(d).$$
We shall find a good map extending $\af$ with domain 
$(E(d),\Gamma_{\ca{E}}, \bk_{\ca{E}}(\alpha))$. 
Consider the
$\sigma$-polynomial $H:= f(G)$, that is,
$$H(x)\ =\ \sigma(x) - \frac{f(b)}{\sigma(f(b))}\cdot x.$$
By saturation we
can find $\alpha' \in \bk'$ with $\bar{H}(\alpha')=0$ and a difference field
isomorphism $g_{\res}: \bk_{\ca{E}}(\alpha) \to
\bk_{\ca{E}'}(\alpha')$ that extends $f_{\res}$, sends $\alpha$
to $\alpha'$ and is elementary as a partial map between the
difference fields $\bk$ and $\bk'$. 
Using again Lemma~\ref{fixr} we find  $d' \in K'$
such that $v'(d')=0$ and $H(d')=0$. Since 
$\bar{H}(\bar{d'})=\bar{H}(\alpha')=0$, we can
multiply $d'$ by an element in $K'$ of valuation zero and
fixed by $\sigma$
to assume further that $\bar{d'}=\alpha'$.
Then Lemmas~\ref{extresa} and \ref{ac2} yield a 
good map $\mathbf{g}=(g, f_{\val}, g_{\res})$ where 
$g:E(d) \to E'(d')$ extends $f$ and sends $d$ to
$d'$. The domain 
$(E(d), \Gamma_{\ca{E}}, \bk_{\ca{E}}(\alpha))$  of $\mathbf{g}$ is  small.

\bigskip\noindent
In the extension procedures (3) and (4) the
value group $v(E^\times)$ does not change, so if the domain $\ca{E}$ of $\af$
satisfies Axiom 2, then so does the domain of any extension of $\af$
constructed as in (3) or (4). Also $\Gamma_{\ca{E}}$ does not change in
(3) and (4), but at this stage we can have $\Gamma_{\ca{E}}\ne v(E^\times)$. 
By repeated application of (1)--(4) we can arrange that $\ca{E}$ is workable 
and satisfies Axiom 4. Then by Corollary~\ref{immsat} we can 
arrange that in addition $\ca{E}$ is $\sigma$-henselian.
(Any use of this in what follows will be explicitly indicated.)

\bigskip\noindent
(5) {\em Towards arranging $\Gamma_{\ca{E}}=v(E^\times)$; the case of no torsion
modulo $v(E^\times)$}. 

Suppose $\gamma\in \Gamma_{\ca{E}}$ has no torsion
modulo $v(E^\times)$, that is, $n\gamma\notin v(E^\times)$ for all $n>0$.
Take $a \in \Fix(K)$ such that
$v(a)=\gamma$. Let $i$ be a $\sigma$-lifting 
of the residue difference field $\bk$ to $\ca{K}$. Since $\ac(a)$ is fixed by
$\bar{\sigma}$, $a/i(\ac(a))\in \Fix(K)$ and $v\big(a/i(\ac(a))\big)=\gamma$. 
So replacing $a$ by $a/i(\ac(a))$ we arrange that 
$v(a)=\gamma$ and $\ac(a)=1$. In the same way we obtain $a' \in \Fix(K')$ such
that $v'(a')=\gamma':= f_{\val}(\gamma)$ and $\ac'(a')=1$. Then by
a familiar fact from the valued field context we have an
isomorphism of valued fields $g: E(a) \to E'(a')$ extending
$f$ with $g(a)=a'$. Then $(g, f_{\val},
f_{\res})$ is a good map with small domain $(E(a),\Gamma_{\ca{E}},\bk_{\ca{E}})$;
this domain satisfies Axiom 2 if $\ca{E}$ does.

\bigskip\noindent
(6) {\em Towards arranging $\Gamma_{\ca{E}}=v(E^\times)$; 
the case of prime torsion modulo $v(E^\times)$}. Here we assume that $\ca{E}$ 
satisfies Axiom 2 and is $\sigma$-henselian. 

Let $\gamma\in \Gamma_{\ca{E}}\setminus v(E^\times)$ with
$\ell\gamma \in v(E^\times)$, where $\ell$ is a prime number. 
As $\ca{E}$ satisfies Axiom 2
we can pick $b \in \Fix(E)$ such that $v(b)=\ell\gamma$. 
Since $\ca{E}$ is
$\sigma$-henselian we have a $\sigma$-lifting of its difference residue field 
$\bk_{\ca{E}}$ to $\ca{E}$ and we can use this as in (5) 
to arrange that
$\ac(b)=1$. We shall find $c \in \Fix(K)$ such that $c^\ell=b$ and $\ac(c)=1$. 
As in (5) we have $a \in \Fix(K)$ such that $v(a)=\gamma$ and 
$\ac(a)=1$. Then the polynomial
$P(x):=x^\ell-b/a^\ell$ over $K$ is henselian at $1$. This gives
$u \in K$ such that $P(u)=0$ and $\bar{u}=1$. Now let $c=au$.
Clearly $c^\ell=b$ and $\ac(c)=1$. Note that $\sigma(c)^\ell=b$, hence
$\sigma(c)=\omega c$ where $\omega$ is an $\ell^{th}$-root of unity. Using $\ac(c)=1$
we get $\ac(\omega)=1$, so $\omega=1$, that is, $c\in \Fix(K)$,  
as promised.  Likewise we find $c' \in \Fix(K')$ such that $c'^\ell=f(b)$ 
and $\ac'(c')=1$. Then $\af$ extends easily to a good map with domain 
$(E(c), \Gamma_{\ca{E}},\bk_{\ca{E}})$ sending $c$ to $c'$; this domain satisfies
Axiom 2.

\bigskip\noindent
By iterating (5) and (6) we can assume in the rest of the proof that  
$\Gamma_{\ca{E}}=v(E^\times)$, and we shall do so. This condition is actually
preserved in the earlier extension procedures (3) and (4), as the reader 
may easily verify. Anyway, we can refer from now on to
$\Gamma_{\ca{E}}$ as the {\em value group\/} of $E$. 
Note also that in the extension procedures (5) and (6) the
residue difference field does not change.

\bigskip\noindent
Now let $a \in K$ be given. We want to extend $\af$ to a good map
whose domain is small and contains $a$. At this stage we can assume 
$\bk_{\ca{E}}=\pi(\ca{O}_E)$, $\Gamma_{\ca{E}}=v(E^\times)$, and 
$\ca{E}$ is workable.
Appropriately iterating and alternating the above 
extension procedures
we arrange in addition that $\ca{E}$
satisfies Axiom 4 and $E\langle a \rangle$ is an immediate
extension of $E$. Let $\ca{E} \langle a \rangle$ be the valued
difference subfield of $\ca{K}$ that has $E\langle a \rangle$ as underlying
difference field. By Corollary~\ref{immsat}, $\ca{E} \langle a \rangle$ has a maximal immediate 
valued difference field extension $\ca{E}_1\le \ca{K}$. Then
$\ca{E}_1$ is a maximal immediate extension of $\ca{E}$ as well. 
Applying Corollary~\ref{immsat} 
to $\ca{E}'$ and using 
Theorem~\ref{unique.max.imm.ext}, we can extend
$\af$ to a good map with domain $\ca{E}_1$, construed here as a good
substructure of $\ca{K}$ in the obvious way. It remains to note that 
$a$ is in the underlying difference field of $\ca{E}_1$.
\end{proof}

\bigskip\noindent
{\bf A variant.} At the cost of a purity assumption we can eliminate 
angular component maps in the Equivalence Theorem.
More precisely,  let $\ca{K}, \ca{K}'$ be as before except that we
do not require angular component maps as part of these structures. The notion 
of {\em good substructure\/} of $\ca{K}$ is similarly modified by changing
clause (3) in its definition to: $\bk_{\ca{E}}$ is a difference subfield of 
$\bk$ with $\pi(\ca{O}_E)\subseteq \bk_{\ca{E}}$. In
defining good maps, condition (ii) on
$f_{\text{r}}$ is to be changed to: $f_{\text{r}}(\pi(a))=\pi'(f(a))$ for all $a\in \mathcal{O}_E$, and $f_{\text{r}}$ is elementary as a partial map between the difference fields $\bk$ and $\bk'$.   

\begin{theorem}\label{embed11a} Suppose $\cha(\bk)=0$, $\ca{K}$, $\ca{K}'$ satisfy Axiom $2$ and are $\sigma$-henselian, and $v(E^\times)$ is pure in 
$\Gamma$. Then any good map 
$\ca{E} \to \ca{E}'$ is a partial elementary map between $\ca{K}$ and $\ca{K}'$.
\end{theorem}
\begin{proof} The case $\Gamma=\{0\}$ being trivial, let $\Gamma\ne \{0\}$, and
let $\af: \ca{E} \to \ca{E}'$ be a good map; our task is to show that $\af$ 
is a partial elementary map between $\ca{K}$ and $\ca{K}'$. 
We first arrange 
that the valued difference subfield 
$(E, v(E^\times),\pi(\ca{O}_E);\dots)$ of $\ca{K}$ is $\aleph_1$-saturated
by passing to an elementary extension 
of a suitable many-sorted structure with  
$\ca{K}$, $\ca{K}'$, $\ca{E}$, $\ca{E}'$ and $\af$ as ingredients.
 As in the beginning of the proof of
Theorem~\ref{embed11} we arrange next that
$\ca{K}$ and $\ca{K}'$ are $\kappa$-saturated, where 
$\kappa$ is an uncountable cardinal such that 
$|\bk_{\ca{E}}|$, $|\Gamma_{\ca{E}}| < \kappa$. Then we apply the extension 
procedures (2) and (3)
in the proof of Theorem~\ref{embed11} to arrange that 
$\bk_{\ca{E}}=\pi(\ca{O}_E)$ and 
$\ca{E}$ satisfies Axiom 2, without changing $v(E^\times)$. To simplify notation we identify $\ca{E}$ and $\ca{E}'$ via $\af$; we have to show that then 
$\ca{K} \equiv_{\ca{E}} \ca{K}'$. Since 
$(E, v(E^\times),\pi(\ca{O}_E);\dots)$ is $\aleph_1$-saturated,
Lemmas~\ref{xs1} and ~\ref{xs2} yield 
cross-sections 
$$s_E: v(E^\times) \to \text{Fix}(E)^\times,\quad s: 
\Gamma\to \text{Fix}(K)^\times,\quad s': \Gamma' \to \text{Fix}(K')^\times$$
such that $s$ and $s'$ extend $s_E$. These cross-sections 
induce angular component maps $\ac_E$ on $\text{Fix}(E)$, $\ac$ on $\text{Fix}(K)$, and 
$\ac'$ on $\text{Fix}(K')$, which by Lemma~\ref{acm1} 
extend uniquely to angular component maps
on $\ca{E}$, $\ca{K}$, and $\ca{K}'$. (Here we use that 
$\ca{E}$ satisfies Axiom 2.) 
This allows us to apply Theorem~\ref{embed11}
to obtain the desired conclusion.    
\end{proof}

\end{section}

\begin{section}{Relative Quantifier Elimination}\label{cqe1}

\noindent
Here we derive various consequences of the Equivalence Theorem of 
Section~\ref{et}. We use the symbols $\equiv$ and $\preceq$ for the 
relations of elementary equivalence and being an elementary submodel,
in the setting of many-sorted structures, and ``definable'' means 
``definable with parameters from the ambient structure''. 
Let $\mathcal{L}$ be the 
3-sorted language of valued fields, with 
sorts $\f$ (the field sort), $\v$ (the value group sort), 
and $\re$ (the residue sort). 
We view a valued field $(K, \Gamma, \bk;\dots)$ as an 
$\mathcal{L}$-structure, with $\f$-variables ranging over $K$, 
$\v$-variables
over $\Gamma$, and $\re$-variables over $\bk$. Augmenting $\mathcal{L}$ 
with a function symbol $\sigma$ of sort $(\f,\f)$ gives the language 
$\mathcal{L}(\sigma)$ of valued difference fields, and augmenting it
further with a function symbol $\ac$ of sort $(\f,\re)$ gives the language
$\mathcal{L}(\sigma,\ac)$ of $\ac$-valued difference fields.
In this section
$$\ca{K}=(K, \Gamma, \bk; \dots), \qquad \ca{K}'=(K', \Gamma', \bk'; \dots)$$
are ac-valued difference fields of 
equicharacteristic $0$ that satisfy Axiom $2$ and are $\sigma$-henselian;
they are considered as $\ca{L}(\sigma, \ac)$-structures.

\begin{corollary}\label{comp00}  
$\ca{K} \equiv \ca{K}'$ if and only if $\bk \equiv \bk'$ as difference fields  and  $\Gamma \equiv \Gamma'$ as ordered abelian groups.
\end{corollary}
\begin{proof} The ``only if'' direction is obvious. Suppose $\bk \equiv \bk'$ as difference fields, and $\Gamma \equiv \Gamma'$ as ordered groups. This gives 
good substructures 
$\ca{E}:=(\mathbb{Q},\{0\}, \mathbb{Q})$ of $\ca{K}$, and $\ca{E}':=(\mathbb{Q},\{0\},\mathbb{Q})$ of $\ca{K}'$, and a trivial good map $\ca{E} \to \ca{E}'$. Now apply Theorem~\ref{embed11}.
\end{proof}

\noindent
Thus $\ca{K}$ is elementarily equivalent to the Hahn difference field
$\bk((t^\Gamma))$ with angular component map defined in the beginning of 
Section~\ref{et}. 

\begin{corollary}\label{comp01} Let 
$\ca{E}=(E, \Gamma_E, \bk_E;\dots)$ be a $\sigma$-henselian ac-valued difference subfield of $\ca{K}$ satisfying Axiom $2$ such that $\bk_E\preceq \bk$ as difference fields, and $\Gamma_E \preceq \Gamma$ as ordered abelian groups. Then $\ca{E} \preceq \ca{K}$.
\end{corollary}
\begin{proof} Take an elementary extension $\ca{K}'$ of 
$\ca{E}$. Then $\ca{K}'$ satisfies Axiom $2$, $(E,\Gamma_E, \bk_E)$ is a good substructure of both $\ca{K}$ and
$\ca{K}'$, and the identity on $(E, \Gamma, \bk_E)$ is a good map. Hence by Theorem~\ref{embed11} we have
$\ca{K} \equiv_{\ca{E}} \ca{K}'$. Since $\ca{E} \preceq \ca{K}'$, this gives
$\ca{E}\preceq \ca{K}$.
\end{proof}

\noindent
The proofs of these corollaries use only weak forms of the 
Equivalence Theorem, but now we
turn to a result that uses its full strength: a relative elimination of 
quantifiers 
for the $\ca{L}(\sigma, \ac)$-theory $T$ of 
$\sigma$-henselian ac-valued 
difference fields of equicharacteristic $0$ that satisfy Axiom 2.
We specify that the function symbols $v$ and $\pi$ of $\ca{L}(\sigma, \ac)$ 
are to be 
interpreted as {\em total\/} functions in any $\ca{K}$ as follows: extend 
$v: K^\times \to \Gamma$  to $v: K \to \Gamma$ by $v(0)=0$, and extend 
$\pi: \ca{O} \to \bk$ to $\pi: K \to \bk$ by $\pi(a)=0$ for $a\notin \ca{O}$.  

Let $\mathcal{L}_{\re}$ be the sublanguage of 
$\mathcal{L}(\sigma, \ac)$ involving only the sort $\re$, that is,
$\mathcal{L}_{\re}$ is a copy of the language of difference fields, 
with $\bar{\sigma}$ as the
symbol for the difference operator. Let $\mathcal{L}_{\v}$ be the 
sublanguage of 
$\mathcal{L}(\sigma, \ac)$ involving only the sort $\v$, that is,
$\mathcal{L}_{\v}$ is
the language of ordered abelian groups.
 
Let $x=(x_1, \dots, x_l)$ be 
a tuple of distinct $\f$-variables, 
$y=(y_1, \dots, y_m)$ a tuple of distinct $\re$-variables, 
and $z=(z_1, \dots, z_n)$ a tuple of distinct $\v$-variables. Define a 
{\em special $\re$-formula in $(x,y)$\/} to be an $\ca{L}(\sigma, \ac)$-formula 
$$\psi(x,y)\ :=\ \psi'\big(\ac(q_1(x)),\dots, \ac(q_k(x)),y\big)$$
where $k\in \mathbb{N}$, $\psi'(u_1,\dots, u_k,y)$ is an $\mathcal{L}_{\re}$-formula, 
and $q_1(x),\dots, q_k(x)\in \mathbb{Z}[x]$. Also, a {\em special $\v$-formula in $(x,z)$\/} is an $\ca{L}(\sigma, \ac)$-formula 
$$\theta(x,z)\  :=\ \theta'\big(v(q_1(x)),\dots, v(q_k(x)),z\big)$$
where $k\in \mathbb{N}$, $\theta'(v_1,\dots, v_k,y)$ is an $\mathcal{L}_{\v}$-formula, 
and $q_1(x),\dots, q_k(x)\in \mathbb{Z}[x]$. Note that these special formulas do not
have quantified $f$-variables. We can now state our relative quantifier
elimination:

\begin{corollary}\label{qe} Every $\ \ca{L}(\sigma, \ac)$-formula
$\ \phi(x,y,z)\ $ is $T$-equivalent to a boolean
combination of special
$\re$-formulas in $(x,y)$ and special $\v$-formulas in $(x,z)$.
\end{corollary}
\begin{proof}
Let $\psi(x,y)$ and $\theta(x,z)$ range over special formulas as described
above. For a model $\ca{K}=(K, \Gamma, \bk;\dots)$ of $T$ and 
$a\in K^l$, $r \in \bk^m$, $\gamma \in \Gamma^n$, let
\begin{align*} \tp_{\re}^{\ca{K}}(a,r) &:= 
\{\psi(x,y):\  \ca{K} \models \psi(a,r)\} \\
  \tp_{\v}^{\ca{K}}(a,\gamma) &:= \{\theta(x,z):\ 
\ca{K} \models \theta(a,\gamma)\}.
\end{align*}
Let $\ca{K}$ and $\ca{K}'$ be any models of $T$, and let
$$(a,r,\gamma)\in K^l\times \bk^m\times \Gamma^n, \qquad (a',r',\gamma')\in K'^l\times \bk'^m\times \Gamma'^n$$ be such that
$\tp^{\ca{K}}_{\re}(a,r)=\tp^{\ca{K}'}_{\re}(a',r')$ and $\tp^{\ca{K}}_{\v}(a,\gamma)=\tp^{\ca{K}'}_{\v}(a',\gamma')$. 
It suffices to show that under these assumptions we have 
$$\tp^{\ca{K}}(a, r, \gamma)\ =\ \tp^{\ca{K}'}(a', r', \gamma').$$
Let $\ca{E}:= (E,\Gamma_{\ca{E}}, \bk_{\ca{E}})$ where 
$E:= \mathbb{Q}\langle a \rangle$, $\Gamma_{\ca{E}}$ is the ordered subgroup of $\Gamma$ generated by
$\gamma$ over $v(E^\times)$, and $\bk_{\ca{E}}$ is the difference subfield of
$\bk$ generated by $\ac(E)$ and $r$, so $\ca{E}$ is a good substructure of
$\ca{K}$. Likewise we define the good substructure $\ca{E}'$ of $\ca{K}'$.
For each $q(x)\in \mathbb{Z}[x]$ we have $q(a)=0$ iff $\ac(q(a))=0$, and also 
$q(a')=0$ iff $\ac'(q(a'))=0$. 
In view of this fact, the assumptions give us a good map $\ca{E} \to \ca{E}'$ 
sending $a$ to
$a'$, $\gamma$ to $\gamma'$ and $r$ to $r'$. It remains to apply 
Theorem~\ref{embed11}.
\end{proof}

\noindent
In the proof above it is important that 
our notion of a good substructure 
$\ca{E}=(E, \Gamma_{\ca{E}}, \bk_{\ca{E}})$ 
did not require $\Gamma_{\ca{E}}=v(E^\times)$ or $ \bk_{\ca{E}}=\pi(\ca{O}_E)$.
This is a difference with the treatment in \cite{BMS}. Related to it 
is that in Corollary~\ref{qe} we have a separation of 
$\re$- and $\v$-variables;  this
makes the next result almost obvious. 

\begin{corollary}\label{comp02} Each subset of $\bk^m\times \Gamma^n$ definable in $\ca{K}$ is a finite union of rectangles $X\times Y$ with $X\subseteq \bk^m$
definable in the difference field $\bk$ and $Y\subseteq \Gamma^n$ definable in 
the
ordered abelian group $\Gamma$.
\end{corollary}
\begin{proof} By Corollary~\ref{qe} and using its notations it is enough to
observe that for $a\in K^l$, a
special $\re$-formula $\psi(x,y)$ in $(x,y)$, and a
special $\v$-formula $\theta(x,z)$in $(x,z)$, 
the set $\{r\in \bk^m: \ca{K}\models \psi(a,r)\}$ is definable
in the difference field $\bk$, and the set $\{\gamma\in \Gamma^n: \ca{K}\models \theta(a,\gamma)\}$ 
is definable in the ordered abelian group $\Gamma$.   
\end{proof}

\noindent
Corollary~\ref{comp02} says in particular 
that the relations on $\bk$ definable in 
$\ca{K}$ are
definable in the difference field $\bk$, and likewise, the relations on 
$\Gamma$ definable in $\ca{K}$ are definable in the ordered abelian group 
$\Gamma$. Thus $\bk$ and $\Gamma$ are stably embedded in $\ca{K}$. 
The corollary says in addition that
$\bk$ and $\Gamma$ are orthogonal in $\ca{K}$.

\medskip\noindent
By Corollary~\ref{acm2} we can get rid of angular component
maps in Corollaries~\ref{comp00} and
\ref{comp02}: these go through if we replace ``ac-valued'' by ``valued''. 
Also Corollary~\ref{comp01} goes through with this change, but for this we
need Theorem~\ref{embed11a}. 
In particular, any $\sigma$-henselian valued difference field 
satisfying Axiom $2$,
with residue difference field $\bk$ of characteristic $0$ and  
value group $\Gamma$, is elementarily equivalent to the Hahn difference field
$\bk((t^\Gamma))$. 

\end{section}

\begin{section}{The unramified mixed characteristic case}\label{mcc}

\noindent
We now aim for mixed characteristic analogues of Sections~\ref{et} and 
~\ref{cqe1}. Kochen~\cite{koch} has a clear 
account how a result like Corollary~\ref{comp00} for henselian valued fields 
can be obtained in mixed characteristic
from the equicharacteristic zero case by coarsening.
We follow here the same track, but to get the mixed characteristic 
Equivalence Theorem~\ref{embed13} we use an elementary fact (Lemma~\ref{lemm1}) 
in a way that may be new and yields a proof that differs from the 
rather complicated treatment in \cite{BMS}.  

\bigskip\noindent
{\bf A better equivalence theorem.}
We first improve Theorem~\ref{embed11} by allowing extra structure 
on the residue difference field and on the value group. 

Let $\mathcal{L}$ be the 3-sorted language of valued fields 
and $\mathcal{L}(\sigma,\ac)$ the language of $\ac$-valued difference fields, 
as 
introduced in Section~\ref{cqe1}. Consider now a language $\mathcal{L}^*\supseteq \mathcal{L}(\sigma,\ac)$
such that every symbol of $\mathcal{L}^*\setminus \mathcal{L}(\sigma,\ac)$
is a relation symbol of some sort 
$(\v,\dots, \v)$ or $(\re,\dots, \re)$.
Let $\mathcal{L}_{\v}^*$ be the sublanguage of 
$\mathcal{L}^*$ involving only the sort $\v$, that is, the language of 
ordered abelian groups
together with the new relation symbols of sort 
$(\v,\dots, \v)$. 
Also, let $\mathcal{L}_{\re}^*$ be the sublanguage of 
$\mathcal{L}^*$ involving only the sort $\re$, that is, 
(a copy of) the language of difference fields 
together with the new relation symbols of sort
$(\re,\dots, \re)$.
(The difference operator symbol of $\mathcal{L}_{\re}^*$ is
$\bar{\sigma}$, to avoid confusion with the difference operator symbol 
$\sigma$ of sort $(\f,\f)$.) 
By a $*$-valued difference field we mean an
 $\mathcal{L}^*$-structure whose $\mathcal{L}(\sigma,\ac)$-reduct is an
$\ac$-valued difference field. 

\medskip\noindent 
Let $\ca{K}=(K,\Gamma, \bk;\cdots)$ be a 
$*$-valued difference field. Then we shall view $\Gamma$ as an
$\mathcal{L}_{\v}^*$-structure and $\bk$ as an
$\mathcal{L}_{\re}^*$-structure, in the obvious way.  
Any subfield $E$ of $K$ is
viewed as a valued subfield of $\ca{K}$ with valuation ring
$\ca{O}_E:=\ca{O}\cap E$.

\medskip\noindent
A {\em good substructure\/} of 
$\ca{K}=(K, \Gamma, \bk;\cdots)$ is a triple
$\ca{E}=(E, \Gamma_{\ca{E}}, \bk_{\ca{E}})$ such that \begin{enumerate}
\item $E$ is a difference subfield of $K$, 
\item $\Gamma_{\ca{E}} \subseteq \Gamma$ as $\mathcal{L}_{\v}^*$-structures
with $v(E^\times)\subseteq \Gamma_{\ca{E}}$,
\item $\bk_{\ca{E}} \subseteq \bk$ as $\mathcal{L}_{\re}^*$-structures with 
$\ac(E)\subseteq \bk_{\ca{E}}$.
\end{enumerate}
In the rest of this subsection 
$\ca{K}=(K, \Gamma, \bk; \cdots)$ and
$\ca{K}'=(K',\Gamma', \bk'; \cdots)$
are $*$-valued difference fields, and
$\ca{E}=(E,\Gamma_{\ca{E}}, \bk_{\ca{E}})$, $\ca{E'}=(E',\Gamma_{\ca{E}'},
\bk_{\ca{E}'})$ are good substructures of
$\ca{K}$, $\ca{K'}$ respectively. 

A {\em good map\/} $\mathbf{f}: \ca{E} \to \ca{E'}$
is a triple $\mathbf{f}=(f, f_{\val}, f_{\res})$ consisting
of an isomorphism $f:E \to E'$ of difference fields, 
an isomorphism $f_{\val}:\Gamma_{\ca{E}} \to
\Gamma_{\ca{E}'}$ of $\mathcal{L}_{\v}^*$-structures and an isomorphism $f_{\res}: \bk_{\ca{E}} \to
\bk_{\ca{E}'}$ of $\mathcal{L}_{\re}^*$-structures such that
\begin{enumerate}
\item[(i)] $f_{\val}(v(a))=v'(f(a))$ for all $a \in E^\times$, and
$f_{\val}$ is elementary as a partial map between the
$\mathcal{L}_{\v}^*$-structures $\Gamma$ and $\Gamma'$; 
\item[(ii)]
$f_{\res}(\ac(a))=\ac'(f(a))$ for all $a \in E$, and
$f_{\res}$ is elementary as a partial map between the
$\mathcal{L}_{\re}^*$-structures $\bk$ and $\bk'$.
\end{enumerate}

\noindent
Theorem~\ref{embed11} goes through in this enriched setting, with the 
same proof except for obvious changes:\footnote{We were told that
Theorem~\ref{embed12} and its corollaries are also 
formal consequences of Theorem~\ref{embed11} and the stable embeddedness and
orthogonality coming from Corollary~\ref{comp02}.}

\begin{theorem}\label{embed12} If $\cha(\bk)=0$ and $\ca{K}$, $\ca{K}'$ 
satisfy Axiom $2$ and are $\sigma$-henselian, then any good map
$\ca{E} \to \ca{E}'$ 
is a partial elementary map between $\ca{K}$ and $\ca{K}'$.
\end{theorem}

\noindent
The four corollaries of Section 8 also go through in this enriched setting,
with residue difference fields and value groups replaced by their 
$\mathcal{L}_{\re}^*$-expansions and 
$\mathcal{L}_{\v}^*$-expansions, respectively. 
In the notions used in Corollary~\ref{qe} the roles of 
$\mathcal{L}_{\re}$ and $\mathcal{L}_{\v}$ are of course taken over by
by $\mathcal{L}_{\re}^*$ and $\mathcal{L}_{\v}^*$, respectively. Except for 
obvious changes the proofs are the same as in Section 8, using 
Theorem~\ref{embed12} in place of Theorem~\ref{embed11}.

\bigskip\noindent
{\bf A variant.} In dealing with the mixed characteristic case it is useful
to eliminate angular component maps in Theorem~\ref{embed12}. So
 let $\ca{K}, \ca{K}'$ be as in the previous subsection 
except that we
do not require angular component maps as part of these structures. The notion 
of {\em good substructure\/} of $\ca{K}$ is then modified by replacing
in clause (3) of its definition the condition $\ac(E)\subseteq \bk_{\ca{E}}$
by $\pi(\ca{O}_E)\subseteq \bk_{\ca{E}}$. In
defining the notion of a good map 
$\af=(f, f_{\text{v}}, f_{\text{r}}): \ca{E} \to \ca{E}'$ the condition on
$f_{\text{r}}$ is to be changed to: $f_{\text{r}}(\pi(a))=\pi(f(a))$ for all $a\in \mathcal{O}_E$, and $f_{\text{r}}$ 
is elementary as a partial map between the $\mathcal{L}_{\re}^*$-structures $\bk$ and $\bk'$.  Then the same arguments as we used in proving
Theorem~\ref{embed11a} yield the following:

\begin{theorem}\label{embed12a} If $\cha(\bk)=0$ and $\ca{K}$ and $\ca{K}'$ satisfy Axiom $2$ and are $\sigma$-henselian, 
and $\ca{E}$ and $\ca{E}'$ are good substructures of $\ca{K}$ and $\ca{K}'$,
respectively, with $v(E^{\times})$ pure in $\Gamma$, then 
any good map 
$\ca{E} \to \ca{E}'$ is a partial elementary map between $\ca{K}$ and $\ca{K}'$.
\end{theorem}

\medskip\noindent
{\bf Coarsening.} To reduce the mixed characteristic case to the equal 
characteristic zero case we use coarsening. In this subsection
$\ca{K}=(K, \Gamma, \bk;\dots)$ is a valued difference field. 
Let $\Delta$ be a convex subgroup of
$\Gamma$, let $\dot{\Gamma}:= \Gamma/\Delta$ be the ordered quotient group, 
and let $\dot{v}:K^\times\to \dot{\Gamma}$ be the composition
$K^\times \to \Gamma \to \dot{\Gamma}$ of $v$ with the
canonical map $\Gamma\to\dot{\Gamma}$, so $\dot{v}$ is again a valuation.
Let $\dot{\ca{O}}$ be the valuation ring of $\dot{v}$, and $\dot{\fr{m}}$ its
maximal ideal, so
\begin{eqnarray*}
\dot{\ca{O}}&=&\{x\in K:\ v(x)\geq\delta,\ \mbox{for some}\ \delta\in\Delta\}\ \supseteq\  \ca{O}:= \ca{O}_v,\\
\dot{\fr{m}}&=&\{x\in K:\ v(x)>\Delta\}\ \subseteq\ \fr{m}.
\end{eqnarray*}
Let $\dot{\bk}=\dot{\ca{O}}/\dot{\fr{m}}$ be the residue field for $\dot{v}$
and let $\dot{\pi}: \dot{\ca{O}} \to \dot{\bk}$ be the canonical map. 
This gives a valued difference field 
$\dot{\ca{K}}:=(K,\dot{\Gamma}, \dot{\bk};\ \dot{v}, \dot{\pi})$
satisfying Axiom 1. Some other axioms are also preserved:

\begin{lemma}\label{coarse1} If $\ca{K}$ satisfies Axiom $2$, so
does $\dot{\ca{K}}$. 
If $\ca{K}$ satisfies Axiom $2$ and is $\sigma$-henselian, then $\dot{\ca{K}}$
is $\sigma$-henselian.
\end{lemma}
\begin{proof} The claim about Axiom 2 is obvious. Assume
$\ca{K}$ satisfies Axiom 2 and is $\sigma$-henselian. 
Let $G(x)$ over $\dot{\ca{O}}$ (of order $\le n$) be $\sigma$-henselian at 
$a\in \dot{\ca{O}}$, with respect to $\dot{\ca{K}}$.
It is easy to check that then
$G,a$ is in $\sigma$-hensel configuration with respect to 
$\ca{K}$. Lemma~\ref{hensel-conf} and 
Remark~\ref{scan} then yield
$b\in K$ such that $v(a-b)= 
v\big(G(a)\big)- \min_{|\i|=1}v\big(G_{(\i)}(a)\big)$, 
and thus $\dot{v}(a-b)=\dot{v}\big(G(a)\big)$, as desired.
\end{proof}

\noindent
Let $\dot{\sigma}$ be the automorphism of the field
$\dot{\bk}$ induced by the difference operator $\sigma$ of $\dot{\ca{K}}$. 
The field $\dot{\bk}$ 
carries the valuation $v_\Delta: \dot{\bk}^\times \to \Delta$ given by
$v_\Delta(x+\dot{\fr{m}})=v(x)$ for $x$ a unit of $\dot{\ca{O}}$. 
The valuation ring of $v_\Delta$ is $\dot{\pi}(\ca{O})$, and we have the 
surjective ring morphism $\pi_{\Delta}\ :\ \dot{\pi}(\ca{O}) \to \bk$ given by
$\pi_{\Delta}\big(\dot{\pi}(a)\big)= \pi(a)$ for all $a\in \ca{O}$. Note that
$$\big((\dot{\bk},\dot{\sigma}), \Delta, \bk;\ v_\Delta,\pi_\Delta\big)$$
is a valued difference field satisfying Axiom 1 with $\dot{\sigma}$
inducing on the residue field $\bk$ the same automorphism $\bar{\sigma}$ as
the difference operator $\sigma$ of $\ca{K}$ does.
The following is now immediate:

\begin{lemma}\label{coarse2} If $\ca{K}$ satisfies Axiom $3$,
so does $\dot{\ca{K}}$.
\end{lemma}

\noindent
Let $\dot{\bk}(*)$ be the expansion 
$(\dot{\bk},\dot{\sigma}, \dot{\pi}(\ca{O}))$ of the 
difference field $(\dot{\bk},\dot{\sigma})$, and let 
$\dot{\ca{K}}(*)$ be the corresponding
expansion $(K, \dot{\Gamma}, \dot{\bk}(*);\ \dot{v}, \dot{\pi})$ of
$\dot{\ca{K}}$. Note that $\ca{O}$ is definable in the structure 
$\dot{\ca{K}}(*)$ by a formula that does not depend on $\ca{K}$:
$$\ca{O}=\{a\in \dot{\ca{O}}:\  \dot{\pi}(a) \in \dot{\pi}(\ca{O})\}.$$
In this way we reconstruct $\ca{K}$ from $\dot{\ca{K}}(*)$.
The advantage of working with $\dot{\ca{K}}(*)$ is that it has
equicharacteristic $0$ if $\ca{K}$ has mixed characteristic and 
$v(p)\in \Delta$.

\bigskip\noindent
Let now $\ca{K}$ be unramified with $\cha{K}=0,\ \cha{\bk}=p>0$. Then
$\Z\cdot v(p)$ is a convex subgroup of $\Gamma$. We set 
$\Delta:=\Z \cdot v(p)$ and 
note that then $\cha{\dot{\bk}}=0$. With these assumptions we have: 

\begin{lemma}\label{coarse3} If $\ca{K}$ is workable, so is $\dot{\ca{K}}$.
\end{lemma}
\begin{proof} Suppose $\ca{K}$ is workable. Then either $\ca{K}$ satisfies Axioms 2 and 3, or it satisfies Axiom 2 and is a Witt case with infinite $\bk$.
In the first case, $\dot{\ca{K}}$
also satisfies Axioms 2 and 3 by Lemmas~\ref{coarse1} and ~\ref{coarse2}, and is
thus workable. It remains to consider 
the case that $\ca{K}$ satisfies Axiom 2 and is a Witt case with infinite 
$\bk$.
Then $\dot{\ca{K}}$
satisfies Axiom 2 by Lemma~\ref{coarse1}, and because $\bk$ is infinite and
$\bar{\sigma}$ is the frobenius map, we have $\bar{\sigma}^d\ne \text{id}$ 
for all $d>0$, and thus
$\dot{\sigma}^d\ne \text{id}$ for all $d>0$. So
$\dot{\ca{K}}$
satisfies Axiom 3 as well.   
\end{proof}

\medskip\noindent
Keeping the assumptions preceding Lemma~\ref{coarse3}, assume also that
$\ca{K}$ is $\aleph_1$-saturated and $\bk$ is perfect. Then the saturation 
assumption guarantees that $\dot{\pi}(\ca{O})$ 
is a complete discrete valuation ring of $\dot{\bk}$. Since $\bk$ is perfect,
this gives a unique ring isomorphism
$\iota: \operatorname{W}[\bk]\ \cong\  \dot{\pi}(\ca{O})$ such that 
$\pi_{\Delta}\circ \iota :\operatorname{W}[\bk]\to \bk$ is the projection map
$(a_0, a_1, a_2,\dots) \mapsto a_0$. Denote the extension of $\iota$ to a 
field isomorphism $\operatorname{W}(\bk)\ \cong\ \dot{\bk}$ also by $\iota$.
If $\ca{K}$ is a Witt case, this gives an isomorphism 
$(\iota,\dots)$ of the Witt difference field 
$\operatorname{W}(\bk)$ onto
$(\dot{\bk}, \Delta, \bk;\ v_\Delta,\pi_\Delta)$.

\bigskip\noindent
{\bf Two lemmas.} The proof of the next lemma uses mainly the 
functoriality of $\operatorname{W}$.

\begin{lemma}\label{lemm1} Let $\bk_0$ be a perfect field with 
$\cha(\bk_0)=p>0$, let $\bk$ and $\bk'$ be 
perfect extension fields of $\bk_0$, and let $\sigma$ and $\sigma'$ be 
automorphisms of $\bk$ and $\bk'$, respectively. Let $\kappa$ be an 
uncountable cardinal such that the difference fields 
$(\bk,\sigma)$ and $(\bk', \sigma')$
are $\kappa$-saturated, $|\bk_0| < \kappa$, and 
$(\bk, \sigma) \equiv_{\bk_0} (\bk', \sigma')$. Then, as difference rings, 
$$ \big(\operatorname{W}[\bk], \operatorname{W}[\sigma]\big)\ 
\equiv_{\operatorname{W}[\bk_0]}\ 
\big(\operatorname{W}[\bk'], \operatorname{W}[\sigma']\big).$$
\end{lemma}
\begin{proof} We just apply the functor
$\operatorname{W}$ to a suitable back-and-forth system between 
$(\bk, \sigma)$ and $(\bk', \sigma')$. In detail, let 
$(\bk_1, \sigma_1)$ range over the difference subfields of
$(\bk,\sigma)$ such that $\bk_0 \subseteq \bk_1$ and  
$|\bk_1| < \kappa$, and let $(\bk_2, \sigma_2)$ range over the 
difference subfields of
$(\bk',\sigma')$ such that $\bk_0 \subseteq \bk_2$ and  
$|\bk_2| < \kappa$. Let $\Phi$ be the set of
all difference field isomorphisms 
$\phi: (\bk_1, \sigma_1) \to (\bk_2, \sigma_2)$ that are the identity on
$\bk_0$ and are partial elementary maps between $(\bk,\sigma)$ and 
$(\bk', \sigma')$.
Note that some $\phi\in \Phi$ maps the definable 
closure of $\bk_0$ in $(\bk, \sigma)$ onto the
definable closure of $\bk_0$ in $(\bk', \sigma')$, so $\Phi\ne \emptyset$ 
and $\Phi$ is a back-and-forth system between
$(\bk, \sigma)$ and  $(\bk', \sigma')$.
The functorial properties of $\operatorname{W}$ and $\kappa$-saturation 
yield a back-and-forth
system $\operatorname{W}[\Phi]$ between  $\big(\operatorname{W}[\bk], \operatorname{W}[\sigma]\big)$ and $\big(\operatorname{W}[\bk'], \operatorname{W}[\sigma']\big)$ consisting of the 
$$\operatorname{W}[\phi]\ :\ 
\big(\operatorname{W}[\bk_1], \operatorname{W}[\sigma_1]\big) \to 
\big(\operatorname{W}[\bk_2], \operatorname{W}[\sigma_2]\big)$$
with $\phi: (\bk_1, \sigma_1) \to (\bk_2, \sigma_2)$ an element of $\Phi$. 
\end{proof}

\noindent
A similar use of functoriality gives:

\begin{lemma}\label{lemm2} Let $\Gamma_0$ be an ordered abelian group 
with smallest positive
element $1$, let $\Gamma$ and $\Gamma'$ be ordered abelian 
extension groups of $\Gamma_0$ with the same smallest positive element $1$.
Let $\kappa$ be an uncountable cardinal such that  
$\Gamma$ and $\Gamma'$
are $\kappa$-saturated, $|\Gamma_0| < \kappa$, and 
$\Gamma \equiv_{\Gamma_0} \Gamma'$. Let $\Delta$ be the common convex 
subgroup $\Z\cdot 1$ of $\Gamma_0, \Gamma$ and $\Gamma'$. Then the ordered
quotient groups 
$\dot{\Gamma}:=\Gamma/\Delta $ and  $\dot{\Gamma}':=\Gamma'/\Delta$ are 
elementarily equivalent over their common ordered subgroup 
$\dot{\Gamma}_0:=\Gamma_0/\Delta$.
\end{lemma}

\bigskip\noindent
{\bf Equivalence in mixed characteristic.} In this final subsection we 
fix a prime number $p$, and
$\ca{K}=(K, \Gamma, \bk;\ v, \pi)$ is a $\sigma$-henselian valued 
difference field such that $\cha(K)=0$, $\bk$ is perfect with $\cha(\bk)=p$, 
and $v(p)$ is the smallest positive element of $\Gamma$. Moreover, assume
either that $\bk$ is infinite and $\bar{\sigma}(x)=x^p$ for all $x\in \bk$
(the Witt case), or that $\bk$ satisfies Axiom 2. In particular, $\ca{K}$ is
workable and $\ca{K}$ is not equipped here with an angular component map. 

We make the corresponding
assumptions about  $\ca{K}'=(K', \Gamma', \bk';\ v', \pi')$. Also, assume that
$\ca{E}=(E,\Gamma_{\ca{E}}, \bk_{\ca{E}})$ and 
$\ca{E}'=(E',\Gamma_{\ca{E'}}, \bk_{\ca{E'}})$ 
are good substructures of 
$\ca{K}$ and $\ca{K}'$, respectively, in the $\ac$-free sense specified 
at the end of Section~\ref{et}, where we defined the corresponding 
$\ac$-free notion of a 
{\em good map\/} $\ca{E} \to \ca{E}'$.
Theorem~\ref{embed11a} goes through in the present setting:

\begin{theorem}\label{embed13} Suppose that $v(E^{\times})$ is pure in 
$\Gamma$ and $\af: \ca{E} \to \ca{E}'$ is a good map. Then $\af$ 
is a partial elementary map between $\ca{K}$ and $\ca{K}'$.
\end{theorem}
\begin{proof} We first arrange that 
$\ca{K}$ and $\ca{K}'$ are $\kappa$-saturated, where 
$\kappa$ is an uncountable cardinal such that 
$|\bk_{\ca{E}}|$, $|\Gamma_{\ca{E}}| < \kappa$. To simplify notation we identify
$\ca{E}$ and $\ca{E}'$ via $\af$, so $\af$ becomes the identity on $\ca{E}$.
We have to show that then $\ca{K}\equiv_{\ca{E}}\ca{K}'$. With $v(p)=1$ as the smallest positive element of $\Gamma_0:= \Gamma_{\ca{E}}$
and of $\Gamma$ and $\Gamma'$ and using the notations of Lemma~\ref{lemm2} 
we have
$\dot{\Gamma} \equiv_{\dot{\Gamma}_0} \dot{\Gamma}'$ by that same lemma.
From the purity of $v(E^{\times})$ in $\Gamma$ it follows that 
$\dot{v}(E^{\times})$ is pure in $\dot{\Gamma}$. 
Since $\ca{K}$ and $\ca{K}'$ are $\aleph_1$-saturated, it is harmless to 
identify $\dot{\bk}$ and $\dot{\bk}'$ with
the fields $\operatorname{W}(\bk)$ and $\operatorname{W}(\bk')$, respectively.
Then the respective valuation rings $\dot{\pi}(\ca{O})$ and 
$\dot{\pi}'(\ca{O}')$ of $\dot{\bk}$ and $\dot{\bk}'$ 
are $\operatorname{W}[\bk]$ and $\operatorname{W}[\bk']$,
and we have the common subfield $\dot{\bk}_{\ca{E}}:=\operatorname{W}(\bk_{\ca{E}})$ of $\dot{\bk}$ 
and $\dot{\bk}'$. It now follows from 
Lemma~\ref{lemm1} that $\dot{\bk}(*)\ \equiv_{\dot{\bk}_{\ca{E}}}\ \dot{\bk}'(*)$.  
Hence the assumptions of Theorem~\ref{embed12a} are satisfied with
$\dot{\ca{K}}(*)$, and  $\dot{\ca{K}}'(*)$ in the role of $\ca{K}$ and 
$\ca{K}'$,
and $\dot{\ca{E}}(*):=(E, \dot{\Gamma}_0,\dot{\bk}_{\ca{E}})$ in the role of both $\ca{E}$ and $\ca{E}'$.
and with
the identity on $\dot{\ca{E}}(*)$ as a good map. This theorem therefore gives 
$$\dot{\ca{K}}(*)\  \equiv_{\dot{\ca{E}}(*)} \dot{\ca{K}}'(*) .$$ 
This yields $\ca{K}\equiv_{\ca{E}}\ca{K}'$ by what we observed just after
Lemma~\ref{coarse2}. 
\end{proof}

\begin{corollary}\label{comp06} 
$\ca{K} \equiv \ca{K}'$ if and only if $\bk \equiv \bk'$ as 
difference fields  and  $\Gamma \equiv \Gamma'$ as ordered abelian groups.
\end{corollary}
\begin{proof} The ``only if'' direction is obvious. Suppose $\bk \equiv \bk'$ as difference fields, and $\Gamma \equiv \Gamma'$ as ordered groups. Then 
we have good substructures 
$\ca{E}:=(\mathbb{Q},\Z, \mathbb{F}_p)$ of $\ca{K}$, and $\ca{E}':=(\mathbb{Q},\Z,\mathbb{F}_p)$ of $\ca{K}'$, and 
an obviously good map $\ca{E} \to \ca{E}'$. Now apply Theorem~\ref{embed13}.
\end{proof}

\noindent
In particular, any $\sigma$-henselian Witt case valued difference 
field satisfying Axiom 2, with infinite residue field $\bk$ and value group $\Gamma\equiv \mathbb{Z}$ as ordered 
abelian groups, is elementarily equivalent to the Witt difference field
$\operatorname{W}(\bk)$. The next result follows from Theorem~\ref{embed13} in the same way as Corollary~\ref{comp01} from Theorem~\ref{embed11}.

\begin{corollary}\label{comp07} Let 
$\ca{E}=(E, \Gamma_E, \bk_E;\dots)$ be a $\sigma$-henselian valued difference subfield of $\ca{K}$ satisfying Axiom $2$ such that $\bk_E\preceq \bk$ as difference fields, and $\Gamma_E \preceq \Gamma$ as ordered abelian groups. Then $\ca{E} \preceq \ca{K}$.
\end{corollary}

\noindent
Theorem~\ref{embed13} does not seem to give a nice relative quantifier elimination such as Corollary~\ref{qe} but it does yield the analogue of Corollary~\ref{comp02}:

\begin{corollary}\label{comp08} Each subset of $\bk^m\times \Gamma^n$ that is definable in $\ca{K}$ is a finite union of rectangles $X\times Y$ with $X\subseteq \bk^m$
definable in the difference field $\bk$ and $Y\subseteq \Gamma^n$ definable in 
the
ordered abelian group $\Gamma$.
\end{corollary}
\begin{proof} By standard arguments we can reduce to the following situation: $\ca{K}$ is $\aleph_1$-saturated, $\ca{E}=(E, \Gamma_E, \bk_E;\dots)\preceq \ca{K}$ is countable, $r, r'\in \bk^m$ have the same type over $\bk_E$, and $\gamma, \gamma'\in \Gamma^n$ have the same type over $\Gamma_E$:
\begin{align*} \tp(r|\bk_E)\ &=\ \tp(r'|\bk_E)  \qquad(\text{in the difference field }\bk),\\
   \tp(\gamma|\Gamma_E)\ &=\ \tp(\gamma'|\Gamma_E) \qquad(\text{in the ordered abelian group }\Gamma).
\end{align*}
It suffices to show that then $(r,\gamma)$ and $(r',\gamma')$ have the same type over
$\ca{E}$ in $\ca{K}$. Let $\bk_1$ and $\bk_1'$ be the definable closures of $\bk_E(r)$
and $\bk_E(r')$ in the difference field $\bk$, and let $\Gamma_1$ and $\Gamma_1'$
be the ordered subgroups of $\Gamma$ generated over $\Gamma_E$ by $\gamma$ and
$\gamma'$. Then $(E, \Gamma_1, \bk_1)$ and $(E, \Gamma_1', \bk_1')$ are good substructures of $\ca{K}$, and the assumption on types yields a good map 
$(E, \Gamma_1, \bk_1)\to (E, \Gamma_1', \bk')$ that is the identity on $(E, \Gamma_E, \bk_E)$, sends $\gamma$ to $\gamma'$ and $r$ to $r'$. Note also that $v(E^\times)=\Gamma_E$ is pure in $\Gamma$. It remains to apply Theorem~\ref{embed13}. 
\end{proof}

\end{section}

\bibliographystyle{plain}

\end{document}